\documentclass[a4paper,10pt, reqno]{amsart}

\usepackage[top=15mm, bottom=15mm, left=4.5mm, right=59mm, marginparwidth=57mm, marginparsep=1mm]{geometry}

\usepackage{amsmath,amssymb,amsthm,enumerate}

\usepackage[svgnames]{xcolor}
\usepackage[colorlinks=true]{hyperref}
\hypersetup{linkcolor=violet,urlcolor=cyan,citecolor=cyan}

\usepackage{pifont}
\usepackage{mathtools}

\usepackage[misc]{ifsym}

\numberwithin{equation}{section}

\theoremstyle{definition}
\newtheorem{theorem}{Theorem}[section]
\newtheorem{prop}[theorem]{Proposition}
\newtheorem{lemma}[theorem]{Lemma}

\theoremstyle{remark}
\newtheorem{remark}[theorem]{Remark}
\newtheorem*{rem*}{Remark}

\newcommand{\cmtr}[1]{{\color{red}  #1 }}

\newcommand\ga{\gamma}

\newcommand\ve{\varepsilon}

\newcommand\De{\Delta}

\newcommand{\R}{\mathbb{R}}

\newcommand{\Z}{\mathbb{Z}}

\newcommand{\pa}{\partial}
\newcommand{\na}{\nabla}

\newcommand{\T} { {\mathbb{T}^2} }

\def\T{\mathbb{T}^2}
\def\Z{\mathbb{Z}}
\def\R{\mathbb{R}}

\usepackage{lipsum}

\makeatletter

\renewcommand\subsection{\@startsection{subsection}{2}%
  \z@{-.5\linespacing\@plus-.7\linespacing}{.5\linespacing}%
  {\normalfont\bfseries}}
\renewcommand\subsubsection{\@startsection{subsubsection}{3}%
  \z@{.5\linespacing\@plus.7\linespacing}{.5\linespacing}%
  {\normalfont\bfseries}}

\makeatother

\begin{document}



\title[EI schemes for the 2D CH equation]{Unconditionally stable exponential integrator schemes for the 2D Cahn-Hilliard equation}

\author[X. Cheng]{Xinyu Cheng}
 \address{X. Cheng, Research Institute of Intelligent Complex Systems, Fudan University, Shanghai, P.R. China}
\email{xycheng@fudan.edu.cn}

\begin{abstract}                
Phase field models are gradient flows with their energy naturally dissipating in time. In order to preserve this property, many numerical
schemes have been well-studied. In this paper we consider a well-known method, namely the
exponential integrator method (EI). In the literature a few works studied several EI schemes for various phase field models and proved the energy dissipation by either requiring a strong Lipschitz condition on the nonlinear source term or certain $L^\infty$ bounds on the numerical solutions (maximum principle). However for phase field models such as the (non-local) Cahn-Hilliard equation, the maximum principle no longer exists. As a result, solving such models via EI schemes remains open for a long time. In this paper we aim to give a systematic approach on applying EI-type schemes to such models by solving the Cahn-Hilliard equation with a first order EI scheme and showing the energy dissipation. In fact second order EI schemes can be handled similarly and we leave the discussion in a subsequent paper. To our best knowledge, this is the
first work to handle phase field models without assuming any strong Lipschitz condition or $L^\infty$ boundedness. Furthermore, we will analyze the $L^2$ error and present some numerical simulations to demonstrate the dynamics.
 
\end{abstract}

\maketitle

\section{Introduction}

 \subsection{Introduction to the models and historical review} 
 Many mathematical models
of physical, biological phenomena can be decribed via partial differential equations (PDEs). Among PDEs, phase field equations are models of essential importance in the area of material sciences. 
In this work we consider two classic phase field models: Allen-Cahn (\ref{AC}) and Cahn-Hilliard (\ref{CH}) equations. The \eqref{AC} model was developed in \cite{OG_AC} by Allen and Cahn to study the competition of crystal grain orientations in an annealing process separation of different metals in a binary alloy; while the \eqref{CH} was introduced in \cite{OG} by Cahn and Hilliard to describe the process of phase separation of different metals in a binary alloy. These equations are presented as: 
\begin{equation} \label{AC}
\left\{
\begin{aligned}
&\partial_t u=\nu \Delta u-f(u),\ \ \  (x,t)\in\Omega\times(0,\infty),\\
&u(x,0)=u_0 ,
\end{aligned}
\right.\ \tag{AC}
\end{equation}
and 
\begin{equation} \label{CH}
\left\{
\begin{aligned}
&\partial_t u=\Delta(-\nu \Delta u+f(u)),\ \ \  (x,t)\in\Omega\times(0,\infty),\\
&u(x,0)=u_0 ,
\end{aligned}
\right.\ \tag{CH}
\end{equation}
where $u(x,t)$ is a real valued function and values of $u$ in $(-1,1)$ represent a mixture of the two phases, with $-1$ representing the pure state of one phase and $+1$ representing the pure state of the other phase. Vector position $x$ is in the spatial domain $\Omega$, which is oftentimes taken to be two or three dimensional periodic domain and $t$ is the time variable. Here $\nu$ is a small parameter, occasionally we denote $\varepsilon=\sqrt{\nu}$ to represent an average distance over which phases mix. The energy term $f(u)$ is often chosen to be
$$f(u)=F'(u)=u^3-u\ ,\ F(u)=\frac14(u^2-1)^2.$$ It is well known that, as $\varepsilon\to 0$, the limiting  problem of \eqref{AC} is driven under a (mean) curvature flow while the limiting problem of \eqref{CH} follows from the Mullins-Sekerka flow; we refer to \cite{AC_MCF} for AC and \cite{Pego}, \cite{ABC} for CH and a recent work for matrix-valued AC \cite{FLWZ23}. Both asymptotic and rigorous analysis are well-studided. This current work will present an idea about how to approach these limiting behaviors in an appropriate way numerically. 

To be more specific, throughout this paper we choose the spatial domain $\Omega$ to be the two dimensional $2\pi$-periodic torus $\T=(\R/2\pi\Z)^2$. In fact, our analysis can be applied to more general settings such as Dirichlet and Neumann boundary conditions in a bounded domain. However, considering the periodic domain allows us to use the efficient and accurate Fourier-spectral numerical algorithms. Beyond this, periodic domain is often appropriate for practical questions, which involve the formation of micro-structure away from physical boundaries. As is well-known that both the AC and CH behave as gradient flows, therefore there energy dissipates in time. Here the associated energy functional is given by 
\begin{align}
    E(u)=\int_\Omega\left(\frac12\nu|\nabla u|^2+F(u)\right)\ dx.\label{energy}
\end{align} Assume that $u(x,t)$ is a smooth solution, it is clear that the energy decays in time: $\frac{d}{dt}E(u(t))\leq0$. Such property thus provides an $a\ priori$ $H^1$-norm bound and since the scaling-critical space for \eqref{CH} is $L^2$ in 2D (and $H^{\frac12}$ in 3D), the global well-posedness follow from standard energy estimates. Therefore from the analysis point of view, the energy dissipation property is an important index for whether a numerical scheme is ``stable'' or not.

There are various approaches existing in the literature that are developed to study numerical behaviors of the Cahn-Hilliard and other related phase field models, cf. \cite{Brian,Keith,Bench19,FengProhl,ShenChen,Bertozzi,E98,BLW22,bL22}. Among which different time stepping approaches are applied including the fully explicit (forward Euler) scheme, fully implicit (backward Euler) scheme, finite element scheme and convex splitting scheme; moreover various schemes are adopted for the spatial discretization including the Fourier-spectral method and finite difference schemes. To guarantee the accuracy and stability, numerical approximations usually need to obey certain qualitative behaviors and features. In this paper we focus on one of the key features, namely the energy dissipation or conservation as mentioned earlier. 
 
Based on the discussion above, we briefly go through the related results on the energy stability in the literature. To start with, Feng and Prohl \cite{FengProhl} introduced a semi-discrete in time and fully spatially discrete finite element method for CH where they obtained an error bound of size of powers of $1/\nu$. However, explicit time-stepping schemes usually require strict time-step restrictions and do not obey energy decay in general (one can compare with Remark~\ref{rem2}). To guarantee the energy dissipation with bigger time steps, a good alternative is to use semi-implicit schemes in which the linear term is implicit (such as backward time differentiation) and the nonlinear term is treated explicitly. Having only a linear implicit at every time step has computational advantages, as suggested in \cite{ShenChen}, where a semi-implicit Fourier-spectral scheme for CH was considered. On the other hand, semi-implicit schemes can lose stability for large time steps and thus smaller time steps are needed in practice. To resolve this problem, semi-implicit methods with better stability have been introduced, cf. \cite{L21,L22,C23a}. Motivated by the work on semi-implicit schemes, the so-called ``exponential integrator schemes'' (EI schemes) start to draw attention in the study of PDEs, cf. \cite{LMS22,RS21}. In a short word, the EI schemes discretize the semi-group presentation of the solutions to nonlinear PDEs. In more details, take the Allen-Cahn equation as an example. It is well-known that the smooth solution to \eqref{AC}  is given by the semi-group mild form (or Duhamel's Formula):
\begin{equation}
    u(t)=e^{\nu t\De}u(t_0)+\int_{t_0}^t e^{\nu (t-s)\De}f(u(s))\ ds,
\end{equation}
where $e^{t\De}$ is the standard heat semi-group (heat kernel). Then the idea of the classic EI is to discretize the semi-group presentation in each time interval $(t_n, t_{n+1})$ by approximating the integrand function by a constant (cf.\cite{LMS22}):
\begin{align*}
    u(t_{n+1})=&e^{\nu \tau\De}u(t_n)+\int_{t_n}^{t_{n+1}} e^{\nu (t_{n+1}-s)\De}f(u(s))\ ds\\
    \approx&e^{\nu \tau\De}u(t_n)+\int_{t_n}^{t_{n+1}} e^{\nu (t_{n+1}-s)\De}f(u(t_n))\ ds\\
    \approx&e^{\nu \tau\De}u(t_n)+\frac{1-e^{\nu(t_{n+1}-t_n)\De}}{\nu\De}f(u(t_n)),
\end{align*}
or equivalently writing $u(t_n)=u^n, u(t_{n+1})=u^{n+1}$ and $\tau=t_{n+1}-t_n$ we have
\begin{align*}
    u^{n+1}=e^{\nu \tau\De}u^n+\frac{1-e^{\nu\tau\De}}{\nu\De}f(u^n).
\end{align*}
However, those works either require a strong Lipschitz condition on the nonlinear source term, or require certain $L^\infty$ bounds on the numerical solutions. Such requirement is satisfied by the Allen-Cahn equation by the standard maximum principle (we refer the readers to \cite{L21,C23a} for example). While for other phase field models such as the (non-local) Cahn-Hilliard equation, the maximum principle no longer exists. As a result, solving phase field models via EI schemes where strong Lipschitz conditions or maximum principle are no longer applicable remains open for a long time. One of our modest goals of this paper is to give a systematic approach on applying EI-type schemes to such models. In particular, we consider one of the classic phase models, the Cahn-Hilliard equation with a first order EI scheme. In fact second order EI schemes can be handled similarly and we leave the discussion in a forthcoming paper. 

In this work we consider the following stabilized EI scheme for \eqref{CH}:
\begin{equation}\label{1.1}
   \left\{\begin{aligned}
&u^{n+1}=e^{-\nu\De^2\tau}u^n+\frac{1-e^{-\nu\De^2\tau}}{\nu\De}\Pi_N [f(u^n)]-S\tau \De^2(u^{n+1}-u^n),
\\ & u^0=\Pi_Nu_0,
 \end{aligned}
 \right.
\end{equation}
where $\tau$ is the time step and $S>0$ is the coefficient for the $O(\tau)$ regularization term. In fact \eqref{1.1} above can be reorganized and rewritten into the following form:
\begin{equation}
\label{1stScheme}
   \left\{\begin{aligned}
&\frac{u^{n+1}-u^{n}}{\tau}+S\De^2(u^{n+1}-u^n)=\frac{e^{-\nu\De^2\tau}-1}{\tau}u^n+\frac{1-e^{-\nu\De^2\tau}}{\tau\nu\De}\Pi_N [f(u^n)],
\\ & u^0=\Pi_Nu_0.
 \end{aligned}
 \right.
\end{equation}

\begin{remark}\label{rem1}
    It is worth mentioning that from \eqref{1stScheme} one can observe that our exponential integrator scheme behaves very similarly to an explicit forward Euler method since 
 \begin{align*}
 &\frac{e^{-\nu\De^2\tau}-1}{\tau}=-\nu\De^2+\frac12 \nu^2\De^4\tau +O(\tau^2),\\
 &\frac{1-e^{-\nu\De^2\tau}}{\tau\nu\De}=\De-\frac12\nu \De^3\tau +O(\tau^2).   
 \end{align*}
Therefore informally speaking, \eqref{1stScheme} can be understood as the usual explicit forward Euler scheme for \eqref{CH} with a stabilization term:
 \begin{align}\label{1.3}
     \frac{u^{n+1}-u^{n}}{\tau}+S\De^2(u^{n+1}-u^n)=-\nu\De^2 u^n+\De\Pi_N [f(u^n)]+O(\nu\tau,\tau^2).
 \end{align}
\end{remark}

\begin{remark}
    The stabilizer $S\De^2(u^{n+1}-u^{n})$ is not unique. In fact our analysis works for many other stabilizers such as $S e^{-\Delta}(u^{n+1}-u^{n})$ and so on. Another particular choice is $\frac{\nu}{2}\De^2 (u^{n+1}-u^n)-S\De(u^{n+1}-u^n)$. Such particular choice results in the analysis of the energy dissipation. More precisely speaking, we need $\frac{\nu}{2}\De^2 (u^{n+1}-u^n)$ to balance the drawback caused by the ``explicit forward'' diffusion; the real stabilizer is indeed $-S\De(u^{n+1}-u^n)$. We choose $S\De^2(u^{n+1}-u^{n})$ just for the sake of simplicity. Even though such choices of stabilizers are not unique, they are necessary. As will be presented in Section 5 with numerical evidence supporting, we will see examples without stabilizers have growing energy. 
    \end{remark}
    
    \begin{remark}
       The exact explicit forward Euler method seems challenging to obtain the energy stability because it seems unlikely to obtain a uniform Sobolev $H^s$-bound estimate. To our best knowledge, we believe such problem arises from the essence of the forward Euler method since the regularity at step $n$ is higher than the regularity at step $n+1$ due to the ``explicit forward'' diffusion $-\De^2 u^n$. This seems to be the crucial technical difference between our exponential integrator scheme \eqref{1stScheme} and the 
    exact explicit forward Euler scheme \eqref{1.3} without the stabilizer $S\De^2(u^{n+1}-u^n)$ and higher order $O(\nu\tau,\tau^2)$ perturbation.
\end{remark}

\begin{remark}\label{rem2}
    As mentioned earlier in Remark~\ref{rem1}, our scheme \eqref{1stScheme} can be understood as an explicit forward Euler scheme. It is well known that after spatial discretization on a grid of size $h$, any explicit scheme will have a time step $\tau$ restriction of size $\tau < O(h^4)$ for the fourth order CH equation. However as will be explained later (see Theorem~\ref{1st_error}), in order to show the convergence of our numerical scheme, we only require $\tau\ll(\log(N))^{-1}$, which is a very mild restriction. Moreover the energy stability holds for any size of $\tau$, see Theorem~\ref{Thm_1stab}.
\end{remark}

\subsection{Main results}

Our main results state below.
\begin{theorem}
[Unconditional energy stability for \ref{CH}]\label{Thm_1stab} Consider \eqref{1.1} or (\ref{1stScheme}) with $\nu>0$ and assume $u_0\in H^2(\T)$. Then there exists a constant $\beta_0$ depending only on the initial energy $E_0=E(u_0)$ such that if 
\begin{equation*}
   \begin{aligned}
S\geq\beta\cdot \left(\Vert u_0\Vert^2_{H^2}+\nu^{-1}|\log \nu|+\nu^{-1}\log N+\nu   \right)
 \end{aligned} ,\quad \beta\geq\beta_0.
\end{equation*}
Then $E(u^{n+1})\leq E(u^n)$, $\forall n\geq 0$ and for any choice of the time step $\tau$, where the energy $E$ is defined in \eqref{energy}.
\end{theorem}

\begin{remark}\label{rem1.6}
    The choice of $S$ is not optimal here; in fact we refer the readers to Section 5 for numerical evidence with much smaller $S$. Determining the optimal bound of the stabilizers can be a very interesting question, however it is not the focus of the presenting paper; we leave this question to the readers.
\end{remark}

\begin{theorem}[$L^2$ error estimate]\label{1st_error}
Fix $\nu>0$ and let $u_0\in H^s$, $s\geq 10$. Let $0<\tau\leq M$ for some $M>0$. Let $u(t)$ be the continuous solution to the 2D Cahn-Hilliard equation with initial data $u_0$. Let $u^m$, $m\geq 1$ be defined in (\ref{1stScheme}) with initial data $u^0$. Assume $S$ satisfies the same condition in Theorem \ref{Thm_1stab}. Define $t_0=0$ and $t_m=m\tau$ for $m\geq 1$. Then for any $m\geq 1$,
\begin{equation*}
\| u^m-u(t_m)\|_2\leq (1+S)\cdot C_2\cdot e^{C_1t_m}\left(N^{-s}+\tau+\tau\cdot N^{-s+4}\right),
\end{equation*}
where $C_1>0$ is a constant depending on $\nu\ ,u_0$; $C_2>0$ is a constant depending on $s\ ,\nu$ and $u_0$.
\end{theorem}
\begin{remark}
It is worth mentioning that the requirement $S\ge \beta \nu^{-1}\log N$ in Theorem~\ref{Thm_1stab} is troublesome in the convergence of the numerical scheme. However as stated in Theorem~\ref{1st_error} above, $S\cdot N^{-s}\to 0$ as $N\to \infty$ therefore such choice of $S$ will not ruin the convergence. On the other hand, we do need $\tau\cdot S\ll 1$, which means $\tau\ll (\log N)^{-1}$. This is still a very mild restriction as discussed in Remark~\ref{rem2} because if one uses spatial discretization, the restriction on the size of the time step $\tau$ is usually much smaller (of polynomial order of $\frac 1N$) to be numerically stable.
\end{remark}

\begin{remark}
Here we need to require that $\tau$ is not arbitrarily large. However, in practice it is not a big issue as we always use small (or at least not arbitrarily large) time steps.
\end{remark}

\subsection{Organization of the presenting paper}
The presenting paper is organized as follows. In Section 2 we list the notation and preliminaries including several useful lemmas. The energy stability of the EI scheme of the 2D Cahn-Hilliard  will be shown in Section 3 while the error estimate is given in Section 4. Numerical experiments will be presented in Section 5. Further discussion and concluding remarks will be in Section 6.



\section{Notation and preliminaries}
Throughout this paper, for any two (non-negative in particular) quantities $X$ and $Y$, we denote $X \lesssim Y$ if
$X \le C Y$ for some constant $C>0$. Similarly $X \gtrsim Y$ if $X
\ge CY$ for some $C>0$. We denote $X \sim Y$ if $X\lesssim Y$ and $Y
\lesssim X$. The dependence of the constant $C$ on
other parameters or constants are usually clear from the context and
we will often suppress  this dependence. We shall denote
$X \lesssim_{Z_1, Z_2,\cdots,Z_k} Y$
if $X \le CY$ and the constant $C$ depends on the quantities $Z_1,\cdots, Z_k$.

For any two quantities $X$ and $Y$, we shall denote $X\ll Y$ if
$X \le c Y$ for some sufficiently small constant $c$. The smallness of the constant $c$ is
usually clear from the context. The notation $X\gg Y$ is similarly defined. Note that
our use of $\ll$ and $\gg$ here is \emph{different} from the usual Vinogradov notation
in number theory or asymptotic analysis.

For a real-valued function $u:\Omega \to \R$ we denote its usual Lebesgue $L^p$-norm by
\begin{align}
    \|u\|_{p}=\|u\|_{L^p(\Omega)}=\begin{cases}
       & \left(\int_{\Omega} |u|^p\ dx\right)^{\frac{1}{p}},\quad  1\le p<\infty;\\
       & \operatorname{esssup}_{x\in\Omega}|u(x)|,\quad p=\infty.
    \end{cases}
\end{align}
Similarly, we use the weak derivative in the following sense: For  $u$, $v\in L^1_{loc}(\Omega)$, (i.e they are locally integrable); $\forall\phi\in C^{\infty}_0(\Omega)$, i.e $\phi$ is infinitely differentiable (smooth) and compactly supported; and 
$$\int_{\Omega}u(x)\ \partial^{\alpha} \phi(x)\ dx=(-1)^{\alpha_1+\cdots+\alpha_n}\int_{\Omega} v(x)\ \phi(x)\ dx ,$$
then $v$ is defined to be the weak partial derivative of $u$, denoted by $\partial^\alpha u$. 
Suppose $u\in L^p(\Omega)$ and all weak derivatives $\partial^\alpha u$ exist for $|\alpha|=\alpha_1+\cdots+\alpha_n \leq k$ , such that $\partial^\alpha u\in L^p(\Omega)$ for $|\alpha|\leq k$, then we denote $u\in W^{k,p}(\Omega)$ to be the standard Sobolev space. The corresponding norm of $W^{k,p}(\Omega)$ is :
$$\| u\|_{W^{k,p}(\Omega)}=\left(\sum_{|\alpha|\leq k}\int_{\Omega}|\partial^\alpha u|^p\ dx\right)^{\frac{1}{p}}\ .$$

\noindent  For $p=2$ case, we use the convention $H^k(\Omega)$ to denote the space $W^{k,2}(\Omega)$. We often use $D^m u$ to denote any differential operator $D^\alpha u$ for any $|\alpha|=m$: $D^2$ denotes $\partial^2_{x_ix_j}u$ for $1\leq i , j\leq d$, as an example.

In this paper we use the following convention for Fourier expansion on $\mathbb{T}^d$: 
$$f(x)=\frac{1}{(2\pi)^d}\sum_{k\in\Z^d}\widehat{f}(k)e^{ik\cdot x}\ ,\ \widehat{f}(k)=\int_{\Omega}f(x)e^{-ik\cdot x}\ dx\ .$$
Taking advantage of the Fourier expansion, we use the well-known equivalent $H^s$-norm and $\dot{H}^s$-semi-norm of function $f$ by $$\| f\|_{H^s}=\frac{1}{(2\pi)^{d/2}}\left(\sum_{k\in\Z^d}(1+|k|^{2s})|\widehat{f}(k)|^2\right)^{\frac12}\ ,\ \| f\|_{\dot{H}^s}=\frac{1}{(2\pi)^{d/2}}\left(\sum_{k\in\Z^d}|k|^{2s}|\widehat{f}(k)|^2\right)^{\frac12}. $$
In addition, for $N\geq 2$, we define
$$X_N=\text{span}\left\{\cos(k\cdot x)\ ,\sin(k\cdot x):\ \ k=(k_1,k_2)\in\Z^2\ ,\ |k|_\infty=\max\{|k_1|,|k_2|\}\leq N \right\} \ .$$


\begin{lemma}[Sobolev inequality on $\mathbb{T}^d$] \label{Sobolevineq}
Let $0<s<d$ and $f\in L^q(\mathbb{T}^d)$ for any $\frac{d}{d-s}<p<\infty$, then
\begin{align*}
\| \left<\nabla\right>^{-s}f\|_{L^p(\mathbb{T}^d)}\lesssim_{s,p,d}\| f\|_{L^q(\mathbb{T}^d)}\ ,\ \mbox{where}\ \frac{1}{q}=\frac{1}{p}+\frac{s}{d}\ ,
\end{align*}
where $\left<\nabla\right>^{-s}$ denotes $(1-\Delta)^{-\frac{s}{2}}$ and $A\lesssim_{s,p,d}B$ is defined as $A\leq C_{s,p,d}\ B$ where $C_{s,p,d}$ is a constant dependent on $s,p$ and $d$.
\end{lemma}

\begin{remark}
Note that the this Sobolev inequality is a variety of the standard version. Note that on the Fourier side the symbol of $\left<\nabla\right>^{-s}$ is given by $(1+|k|^2)^{-\frac{s}{2}}$.  
In particular, $\| f\|_{\infty(\mathbb{T}^d)}\lesssim\| f\|_{H^2(\mathbb{T}^d)}$, known as Morrey's inequality.
\end{remark}

\begin{remark}
    If one further requires that $f\in H^1(\T)$ has zero mean, i.e. $\widehat{f}(0)=0$ we have
    \begin{align}\label{2.4}
        \|f\|_{4}\le \|f\|_2^{\frac12}\|\na f\|_2^{\frac12},
    \end{align}
    and
    \begin{align}\label{2.5}
                \|f\|_{6}\le \|f\|_2^{\frac13}\|\na f\|_2^{\frac23}.
    \end{align}
    The proof of \eqref{2.4} and \eqref{2.5} follows from Lemma~\ref{Sobolevineq} and standard interpolation.
\end{remark}


\begin{lemma}\label{lem2.3}
    Suppose $f\in H^{1}(\T)$ and $f$ has zero mean, i.e. $\widehat{f}(0)=0$. Then
    \begin{align}\label{2.2}
        &\| f\|_2\le \||\na|^{-1} f\|_2^{\frac12} \|\na f\|_2^{\frac12},\\
        & \|f\|_2\le \|\na f\|_2.\label{2.3}
    \end{align}
\end{lemma}
\begin{proof}
    The proof follows from Parseval's identity directly:
    \begin{align*}
        \int_{\T} f\ dx=\int_{\T} |\na|^{-1} f\cdot |\na|f\ dx.
    \end{align*}
    Then \eqref{2.2} follows from the classic Cauchy-Schwarz inequality by noting $\||\na|f\|_2=\|\na f\|_2$. \eqref{2.3} follows from the standard Poincar\'e inequality (or directly from the Fourier side).
 \end{proof}

\begin{lemma}[Log-type interpolation]\label{Loginterpolation}
\noindent For all $ f\in H^s(\T)\ ,\ s>1 $ and suppose $f$ has zero mean, i.e. $\widehat{f}(0)=0$, then \begin{equation*}
\| f\|_\infty\leq C_s\cdot\left( \| f\|_{\dot{H}^1}\sqrt{\log(\| f\|_{\dot{H}^s}+3)}+1\right)\ .
\end{equation*}
\noindent Here $C_s$ is a constant which only depends on s. 
\end{lemma}

\begin{proof}
This lemma is a special case of Lemma 3.1 in \cite{C23a} and we only sketch the proof here. Note that $f$ is mean-zero, we then consider the Fourier series $f(x)=\frac{1}{(2\pi)^2}\sum_{|k|\ge 1}\widehat{f}(k)\ e^{ik\cdot x}$. It then follows that
\begin{equation*}
    \begin{aligned}
\| f\|_{\infty}&\leq\frac{1}{(2\pi)^2}\sum_{|k|\ge 1}|\widehat{f}(k)|\\
&\leq \frac{1}{(2\pi)^2}\left(\sum_{1\le|k|\leq N}|\widehat{f}(k)|+\sum_{|k|>N}|\widehat{f}(k)|\right)\\
&\lesssim \sum_{1\le|k|\leq N}(|\widehat{f}(k)\| k|\cdot|k|^{-1})+\sum_{|k|> N}(|\widehat{f}(k)\| k|^{s}\cdot|k|^{-s})\\
&\lesssim \left(\sum_{1\le|k|\leq N}|\widehat{f}(k)|^2|k|^2 \right)^{\frac12} \cdot \left(\sum_{1\le|k|\leq N}|k|^{-2} \right)^{\frac12}+\left(\sum_{|k|> N}|\widehat{f}(k)|^2|k|^{2s} \right)^{\frac12} \cdot (\sum_{|k|> N}|k|^{-2s})^{\frac12}\\
&\lesssim \frac{1}{N^{s-1}}\left(\sum_{|k|> N}|\widehat{f}(k)|^2|k|^{2s} \right)^{\frac12}+\left(\sum_{1\le|k|\leq N}|\widehat{f}(k)|^2|k|^2\right)^{\frac12}\cdot \sqrt{\log(N+3)}\\
&\lesssim\frac{1}{N^{s-1}}\| f\|_{\dot{H}^s}+\sqrt{\log(N+3)}\| f\|_{\dot{H}^1}\ .
\end{aligned}
\end{equation*}
If $\| f\|_{\dot{H}^s}\leq 3$, we can simply take $N=1$; otherwise take $N^{s-1}$ close to $\| f\|_{\dot{H}^s}$ . 
\end{proof}

\begin{lemma}[Discrete Gr\"{o}nwall's inequality]\label{DiscreteGronwall}
Let $\tau>0$ and $y_n\geq 0$, $\alpha_n\geq 0$, $\beta_n\geq 0$ for $n=1,2,3\cdots $. Suppose
$$\frac{y_{n+1}-y_n}{\tau}\leq \alpha_ny_n+\beta_n\ ,\ \forall\ n\geq 0\ .$$
Then for any $m\geq 1$, we have
\begin{align*}
 y_m\leq \exp\left(\tau\sum_{n=0}^{m-1}\alpha_n\right)\left(y_0+\sum_{k=0}^{m-1}\beta_k\right)\ .
\end{align*}
\end{lemma}
\begin{proof}
We only sketch the proof here for the sake of completeness. By the assumption, it follows that for $n\geq 0$,
$$y_{n+1}\leq (1+\alpha_n\tau)y_n+\tau\beta_n\leq e^{\tau\alpha_n}y_n+\tau\beta_n; $$
therefore one can derive that
$$\exp\left( -\tau\sum_{j=0}^{n}\alpha_j\right)y_{n+1}\leq \exp\left( -\tau\sum_{j=0}^{n-1}\alpha_j\right)y_{n}+\exp\left( -\tau\sum_{j=0}^{n}\alpha_j\right)\beta_n.$$
We thus obtain the desired result by performing a telescoping summation.
\end{proof}

\begin{lemma}[$H^k$ boundedness of the exact CH solution]\label{H^kregularity_CH}
Assume $u(x,t)$ is a smooth solution to the Cahn-Hilliard equation in $\mathbb{T}^2$ and the initial data $u_0\in H^k(\mathbb{T}^2)$ for $k\geq 2$. Then,
\begin{equation}
\sup_{t\geq 0}\| u(t)\|_{H^k(\mathbb{T}^2)}\lesssim_{k} 1
\end{equation}
where we omit the dependence on $\nu$ and $u_0$. 
\end{lemma}
\begin{proof}
    The proof is standard and we postpone the proof to the Appendix.
\end{proof}

\section{Energy stability of the scheme (\ref{1stScheme}) on the 2D Cahn-Hilliard equation }
\label{section:stability2DAC}
Recall that the Cahn-Hilliard equation \eqref{CH} takes the following form:
\begin{align*}\begin{cases} 
\partial_t u=-\nu \Delta^2 u+\De f(u),
\\ u(x,0)=u_0 .
\end{cases}\end{align*}
Here $f(u)=u^3-u$, and the spatial domain $\Omega$ is taken to be the two dimensional $2\pi-$periodic torus $\T$. The corresponding energy is defined by $E(u)=\int_{\Omega}(\frac{\nu}{2} |\nabla u|^2+F(u))\ dx$ , where $F(u)=\frac14 (u^2-1)^2$, the anti-derivative of $f(u)$. Recall that we consider the stabilized exponential integrator scheme \eqref{1stScheme}:
\begin{equation}
\label{1stScheme1}
   \left\{\begin{aligned}
&\frac{u^{n+1}-u^{n}}{\tau}+S\De^2(u^{n+1}-u^n)=\frac{e^{-\nu\De^2\tau}-1}{\tau}u^n+\frac{1-e^{-\nu\De^2\tau}}{\tau\nu\De}\Pi_N f(u^n),
\\ & u^0=\Pi_Nu_0.
 \end{aligned}
 \right.
\end{equation}
Here $\Pi_N$ is truncation of Fourier modes of $L^2$ functions to $|k|_\infty\leq N$. Note that here $E^0=E(\Pi_Nu_0)$ while $E_0=E(u_0)$ and in general $E_0\neq E^0$. In particular the following statement holds.
\begin{lemma} Suppose $E^0=E(\Pi_Nu_0)$ and $E_0=E(u_0)$ as defined above, the following inequality holds:
\begin{equation*}
\sup_N\ E(\Pi_Nu_0)\lesssim 1+E_0, \mbox{where}\ u_0\in H^1({\T})\ .
\end{equation*}
\end{lemma}
\begin{proof}
We rewrite $\Pi_Nu_0$ as $\frac{1}{(2\pi)^2}\sum_{|k|\leq N}\widehat{u_0}(k)e^{ik\cdot x}$ , namely the Dirichlet partial sum of $u_0$. 
\begin{equation*}
\begin{aligned}
\|\nabla\left(\Pi_Nu_0\right)\|_{L^2(\T)}^2=\frac{1}{(2\pi)^2}\sum_{|k|\leq N}|k|^2|\widehat{u_0}(k)|^2\leq \frac{1}{(2\pi)^2}\sum_{|k|\in \Z^2}|k|^2|\widehat{u_0}(k)|^2=\|\nabla\left(u_0\right)\|_{L^2(\T)}^2
\end{aligned}\ .
\end{equation*}
On the potential energy part, by the Sobolev inequality Lemma \ref{Sobolevineq}, $\| u_0\|_{L^4(\T)}\lesssim\| u_0\|_{H^1(\T)}$, this shows $u_0\in L^4(\T)$ and hence the Dirichlet partial sum $\Pi_Nu_0$ converges to $u_0$ in $L^4(\T)$. Then
$\lim_N\|\Pi_Nu_0\|_{L^4(\T)}= \| u_0\|_{L^4(\T)}$, which leads to $\sup_N\|\Pi_Nu_0\|_{L^4(\T)}<\infty$. By the Uniform Boundedness Principle, we derive $\sup_N\|\Pi_N\|<\infty$, i.e. $\sup_N\| \Pi_N u_0\|_{L^4(\T)}\leq c\| u_0\|_{L^4(\T)}$ for an absolute constant $c$. Combining the two estimates above we finish the proof.
\end{proof}

We will prove Theorem \ref{Thm_1stab} by induction. To start with,  let us multiply the equation by $\frac{\nu\De\tau}{e^{-\nu\tau\De^2}-1}(u^{n+1}-u^n)$ and integrate over $\T$. Then the LHS of \eqref{1stScheme1} is
\begin{align}
   \mbox{LHS}= &\frac{1}{\tau}\left(u^{n+1}-u^{n}, \frac{\nu\De\tau}{e^{-\nu\tau\De^2}-1}(u^{n+1}-u^{n})\right)+S\left(\De^2(u^{n+1}-u^{n}), \frac{\nu\De\tau}{e^{-\nu\tau\De^2}-1}(u^{n+1}-u^{n})\right)\\
   =&\left(u^{n+1}-u^{n}, \frac{-\nu\De}{1-e^{-\nu\tau\De^2}}(u^{n+1}-u^{n})\right)+S\left(u^{n+1}-u^{n}, \frac{-\nu\De^3\tau}{1-e^{-\nu\tau\De^2}}(u^{n+1}-u^{n})\right),
\end{align}
where $(\cdot,\cdot)$ is the usual $L^2$ inner product. We observe that 
both the two operators $A=\frac{-\nu\De}{1-e^{-\nu\tau\De^2}}$ and $B=\frac{-\nu\De^3\tau}{1-e^{-\nu\tau\De^2}}$ are positive semi-definite and self-adjoint, therefore we have
\begin{align}\label{3.4}
    \mbox{LHS}=\| \sqrt{A}(u^{n+1}-u^{n})\|_2^2+S\|\sqrt{B}(u^{n+1}-u^{n})\|_2^2,
\end{align}
where the operator $\sqrt{A}$ and $\sqrt{B}$ correspond to the following Fourier symbols:
\begin{align}\label{3.5}
    \widehat{\sqrt{A} u}(k)=\sqrt{\frac{\nu|k|^2}{1-e^{-\nu|k|^4\tau}}}\widehat{u}(k);\quad \widehat{\sqrt{B} u}(k)=\sqrt{\frac{\nu|k|^6\tau}{1-e^{-\nu|k|^4\tau}}}\widehat{u}(k).
\end{align}
To proceed from here, it is important to examine the behaviors of these two operators. Indeed we state the following lemma.
\begin{lemma}\label{lem3.2}
    Suppose the operators $\sqrt{A}$ and $\sqrt{B}$ are defined in \eqref{3.5} above. Then the following inequalities hold for any function $u\in H^3(\T)$ and any $v\in H^1(\T)$ with zero-mean:
    \begin{align}\label{3.6}
   \left(\frac{1}{\tau}\right)^{\frac12} \| |\na|^{-1}v\|_2\le 
 \|\sqrt{A}v\|_2,\quad
        \|\na u\|_2\le \|\sqrt{B}u\|_2.
    \end{align}
\end{lemma}
\begin{proof}
    The \eqref{3.6} can be proved from the Fourier side. In particular, by the Parseval's identity we have
    \begin{align}\label{3.7}
        \|\sqrt{A}v\|_2^2=\frac{1}{(2\pi)^2}\sum_{|k|\ge 1}\frac{\nu |k|^2}{1-e^{-\nu|k|^4\tau}}|\widehat{v}(k)|^2,
    \end{align}
    therefore it suffices to show 
\begin{align}\label{3.8}
    \frac{\nu |k|^2}{1-e^{-\nu|k|^4\tau}}\ge \frac{1}{\tau|k|^2}, \quad\mbox{or}\quad \frac{\nu\tau |k|^4}{1-e^{-\nu\tau|k|^4}}\ge 1.
\end{align}
To show \eqref{3.8}, we can define an 1D function for $x\ge 1$, $h(x)=\frac{\nu\tau x^4}{1-e^{-\nu\tau x^4}}-1$. By direct computation we can show $h'(x)>0$ for $x\ge 1$ and $h(1)>0$ noting $\nu\tau>0$. Therefore \eqref{3.8} is proved. Then \eqref{3.7} implies
\begin{align}
    \frac{1}{\tau}\||\na|^{-1} v\|_2^2\le \|\sqrt{A}v\|_2^2.
\end{align}
By the same reason we can show $\|\na u\|_2\le \|\sqrt{B}u\|_2$. We hereby conclude \eqref{3.6}.
\end{proof}
We then continue the proof of Theorem~\ref{Thm_1stab}. We now focus on the RHS of \eqref{1stScheme1}:
\begin{align}
    \mbox{RHS}=&\left(\frac{e^{-\nu\De^2\tau}-1}{\tau} u^{n},\frac{-\nu\De\tau}{1-e^{-\nu\tau\De^2}}(u^{n+1}-u^{n})\right)+\left(\frac{1-e^{-\nu\De^2\tau}}{\nu\tau\De}\Pi_{N}f(u^n),\frac{-\nu\De\tau}{1-e^{-\nu\tau\De^2}}(u^{n+1}-u^{n})\right)\\
    =&\nu\left(u^{n},\De(u^{n+1}-u^{n})\right)-\left(\Pi_{N} f(u^n),u^{n+1}-u^{n}\right).
\end{align}
Note that $u^{n+1},u^n\in X_N$, we have 
\begin{align}\label{3.12}
    \mbox{RHS}=\nu\left(\na u^n, \na(u^n-u^{n+1})\right)-\left(f(u^n),u^{n+1}-u^{n}\right).
\end{align}
Using the identity $2a(a-b)=a^2-b^2+(a-b)^2$, we derive from \eqref{3.12} that
\begin{align}\label{3.13}
    \mbox{RHS}=\frac{\nu}{2}(\|\na u^n\|_2^2-\|\na u^{n+1}\|_2^2+\|\na u^{n}-\na u^{n+1}\|_2^2)-(f(u^n),u^{n+1}-u^{n}).
\end{align}
To proceed further, by the fundamental theorem of calculus we arrive at
 \begin{equation}\label{3.14}
    \begin{aligned}
    F(u^{n+1})-F(u^n)&=f(u^n)(u^{n+1}-u^n)+\int_{u^n}^{u^{n+1}}f'(s)(u^{n+1}-s)\ ds\\
    &=f(u^n)(u^{n+1}-u^n)+\int_{u^n}^{u^{n+1}}(3s^2-1)(u^{n+1}-s)\ ds\\
    &=f(u^n)(u^{n+1}-u^n)+\frac14(u^{n+1}-u^n)^2\left(3(u^n)^2+(u^{n+1})^2+2u^nu^{n+1}-2\right)\ .
    \end{aligned}
 \end{equation}
Combining the estimates \eqref{3.4}, \eqref{3.13} and \eqref{3.14} we arrive at  
\begin{equation}\label{3.15}
\begin{aligned}
   & \|\sqrt{A}(u^{n+1}-u^{n}\|_2^2+S\|\sqrt{B}(u^{n+1}-u^{n})\|_2^2+E(u^{n+1})-E(u^{n})\\=& \frac{\nu}{2}\|\na(u^{n+1}-u^n)\|_2^2+\frac14\left((u^{n+1}-u^n)^2,3(u^n)^2+(u^{n+1})^2+2u^nu^{n+1}-2\right).
\end{aligned}
\end{equation}
By further considering \eqref{3.6} and \eqref{3.15} we indeed have
 \begin{equation}
     \begin{aligned}
         & \frac{1}{\tau}\||\na|^{-1}(u^{n+1}-u^{n})\|_2^2+(S-\frac{\nu}{2})\|\na(u^{n+1}-u^{n})\|_2^2+E(u^{n+1})-E(u^{n})\\& \le \frac14\left((u^{n+1}-u^n)^2,3(u^n)^2+(u^{n+1})^2+2u^nu^{n+1}-2\right). 
     \end{aligned}
 \end{equation}
Moreover 
by Lemma~\ref{lem2.3}, we derive that
 \begin{equation}
     \begin{aligned}
         & (\sqrt{\frac{2S}{\tau}}+\frac{S}{2}-\frac{\nu}{2})\|u^{n+1}-u^{n}\|_2^2+E(u^{n+1})-E(u^{n}) \le \frac14\left((u^{n+1}-u^n)^2,3(u^n)^2+(u^{n+1})^2+2u^nu^{n+1}-2\right). 
     \end{aligned}
 \end{equation}
Therefore to show the energy dissipation, it suffices to take
\begin{equation}
\label{1stcondtion}
    \begin{aligned}
       \sqrt{\frac{2S}{\tau}}+\frac{S}{2}-\frac{\nu}{2}\geq\frac32\max\left\{\| u^n\|_\infty^2\ ,\ \| u^{n+1}\|_\infty^2\right\}
    \end{aligned}\ .
\end{equation}
We rewrite the numerical scheme \eqref{1stScheme} as follows:
\begin{equation}\label{3.19}
    \begin{aligned}
      u^{n+1}=\frac{S\tau\De^2+e^{-\nu\De^2\tau}}{1+S\tau \De^2}u^n+\frac{1-e^{-\nu\De^2\tau}}{(1+S\tau\De^2)\nu\De}\Pi_N f(u^n)
    \end{aligned}
\end{equation}

\noindent By the interpolation lemma (Lemma \ref{Loginterpolation}), to control $\| u^{n+1}\|_\infty$ and $\| u^n\|_\infty$, we can only focus on the $\dot{H}^1$-norm and $\dot{H}^{\frac32}$-norm.

\begin{lemma}\label{1stLem}
There is an absolute constant $c_1>0$ such that for any $n\geq0$
\begin{align}
\begin{cases}
    \| u^{n+1}\|_{\dot{H}^{\frac32}(\T)}&\leq c_1\cdot \left(\sqrt{\frac{N}{\nu}}+\frac{1}{S} \right)\cdot(E^n+1)  \\
\|u^{n+1}\|_{\dot{H}^1(\T)}&\leq\left(\frac1S+\frac1S\| u^n\|^2_\infty\right)\cdot\|u^n\|_{L^2(\T)}+\|u^n\|_{\dot{H}^1(\T)}\ .
\end{cases}
    \end{align}
\end{lemma}

\begin{proof}
 It suffices to consider Fourier modes $1\le|k|\le N$ from the Fourier side. We then obtain that
\begin{equation}
  \left\{  \begin{aligned}
          &\frac{(S\tau|k|^4+e^{-\nu\tau|k|^4})|k|^{\frac32}}{1+S\tau|k|^4}\le |k|\cdot N^{\frac12}\\
          &  \frac{(1-e^{-\nu\tau|k|^4})|k|^{\frac32}}{(1+S\tau|k|^4)\nu|k|^2}\le\frac{1}{S|k|^{\frac12}} ,
    \end{aligned}
    \right.
\end{equation}
where we have observed that $1-e^{-\nu\tau|k|^4}\le \nu\tau|k|^4$. It then follows that
\begin{equation}  \label{H32norm}
    \| u^{n+1}\|_{\dot{H}^{\frac32}(\T)}\le N^{\frac12}\| u^n\|_{\dot{H}^1(\T)}+\frac{1}{S}\|\big \langle\nabla\big \rangle^{-\frac12}f(u^n)\|_{L^2(\T)}  .
\end{equation}
Here the notaion $\big\langle\nabla\big\rangle^s=(1-\Delta)^{\frac{s}{2}}$ , corresponds to the Fourier side $(1+|k|^2)^{s/2}$. 
Note that by the Sobolev inequality Lemma~\ref{Sobolevineq} we have
\begin{align*}
    \|\big \langle\nabla\big \rangle^{-\frac12}f(u^n)\|_{L^2(\T)}&\lesssim\| f(u^n)\|_{L^{\frac43}(\T)}=\|(u^n)^3-u^n\|_{L^{\frac43}(\T)}\\
    &=\left(\int_{\T}((u^n)^3-u^n)^{\frac43}\ dx\right)^{\frac34}\\&\lesssim E^n+1\ .
\end{align*}
Therefore \eqref{H32norm} can be estimated by the following:
\begin{equation*}
    \| u^{n+1}\|_{\dot{H}^{\frac32}(\T)}\lesssim\left(\sqrt{\frac{N}{\nu}}+\frac{1}{S}\right)(E^n+1).
\end{equation*}
Similarly, we get
\begin{equation*}
    \left\{\begin{aligned}
          & \frac{(S\tau|k|^4+e^{-\nu\tau|k|^4})|k|}{1+S\tau|k|^4}\le |k|
           \\
           &\frac{(1-e^{-\nu\tau|k|^4})|k|}{(1+S\tau|k|^4)\nu|k|^2}\le\frac{\nu\tau|k|^4}{S\tau\nu|k|^6 }|k|=\frac{1}{S|k|}\ .
    \end{aligned}
    \right.
\end{equation*}
This implies for $|k|\ge 1$ that
\begin{equation*}
\begin{aligned}
    \|u^{n+1}\|_{\dot{H}^1(\T)}&\le\|u^n\|_{\dot{H}^1(\T)}+\frac1S\|f(u^n)\|_{L^2(\T)}\\
  &\le\| u^n\|_{\dot{H}^1(\T)}+\frac1S\|(1-(u^n)^2)\cdot u^n\|_{L^2(\T)}\\
  &\le\| u^n\|_{\dot{H}^1(\T)}+(\frac1S+\frac{\| u^n\|^2_\infty}{S})\| u^n\|_{L^2(\T)}.
  \end{aligned}
\end{equation*}
\end{proof}
\begin{proof}[Proof of Theorem~\ref{Thm_1stab}]
Now we will complete the proof for \textbf{Theorem} \ref{Thm_1stab} by induction:

\noindent\textbf{Step 1:} The induction $n\to n+1$ step.
Assuming $E^n\leq E^{n-1}\leq\cdots\leq E^0$ and $E^n\leq \sup_N\ E(\Pi_Nu_0)$, we will show $E^{n+1}\leq E^n$. This implies $\| u^n\|^2_{\dot{H}^1}=\| \nabla u^n\| ^2_{L^2}\leq \frac{2E^n}{\nu}\leq\frac{2E^0}{\nu}$. By Lemma \ref{Loginterpolation} and adopting the notation $f\lesssim_{E^0}g$ to denote that $f\leq C(E^0)\cdot g$ for some constant $C(E^0)$ depending only on $E^0$, we have
\begin{equation}\label{1stInfinity_bound_n}
\begin{aligned}
\| u^n\|^2_\infty&\lesssim\| u^n\|^2_{\dot{H}^1}\left(\sqrt{\log(3+c_1\left(\sqrt{\frac{N}{\nu}}+\frac{1}{S}\right)(E^{n-1}+1))}\right)^2+1\\
&\lesssim\frac{2E^0}{\nu}\left(1+\log(S)+\log(\frac1\nu)+(\log(N))\right)+1\\
&\lesssim_{E^0}\nu^{-1}\left(1+\log(S)+\log(\frac{1}{\nu})+\log(N)\right)+1\ .
\end{aligned}
\end{equation}
Define $m_0 \coloneqq\nu^{-1}\left(1+\log(S)+|\log(\nu)|+\log(N)\right)$, and note that $E^0\leq \sup_N\ E(\Pi_Nu_0)\lesssim E_0+1$, the inequality above (\ref{1stInfinity_bound_n}) is then rewritten as follows:
\begin{equation*}
\| u^n\|^2_\infty\lesssim_{E_0}m_0+1.
\end{equation*}
On the other hand by Lemma \ref{1stLem},
\begin{equation}\label{1stInifinity_bound_n+1}
\begin{aligned}
\| u^{n+1}\|_\infty^2&\lesssim1+\| u^{n+1}\|_{\dot{H}^1}^2\log(3+\| u^{n+1}\|_{\dot{H}^{\frac32}})\\
&\lesssim1+(1+\frac{1+\| u^n\|^2_\infty}{S})^2\| u^n\|_{\dot{H}^1}^2 \log(3+\| u^{n+1}\|_{\dot{H}^{\frac32}})\\
&\lesssim_{E_0}1+(1+\frac{m_0+1}{S})^2\frac{1}{\nu}\log(3+\| u^{n+1}\|_{\dot{H}^{\frac32}})\\
&\lesssim_{E_0} 1+(1+\frac{m_0}{S})^2m_0
\\&\lesssim_{E_0} 1+\frac{m_0^3}{S^2}+m_0 .
\end{aligned}
\end{equation}
The sufficient condition (\ref{1stcondtion}) thus becomes 
\begin{equation*}
\left\{
\begin{aligned}
& \sqrt{\frac{2S}{\tau}}+\frac{S}{2}-\frac{\nu}{2}\ge C(E_0)\left(1+m_0+\frac{m_0^3}{S^2}\right),\\
&m_0=\nu^{-1}\left(1+\log(S)+|\log(\nu)|+\log(N)\right)\ .
\end{aligned}
\right.
\end{equation*}
It then suffices for us to choose $S$ such that
$$S\gg_{E_0}m_0+\nu=\nu^{-1}\left(1+\log(S)+|\log(\nu)|+\log(N)\right)+\nu\ ,$$
where $B\gg_{E_0}D$ means there exists a large constant depending only on $E_0$. In fact, for $\nu\gtrsim1$ , we can take $S\gg_{E_0}\nu+\log(N)$; if $0<\nu\ll 1$, we will choose $S=C_{E_0}\cdot\nu^{-1}(|\log\nu|+\log N)$ , where $C_{E_0}$ is a large constant depending only on $E_0$. Therefore it suffices to choose 
\begin{equation}\label{3.25}
S=C_{E_0}\nu^{-1}(\log N+|\log \nu|)+\nu.
\end{equation}

\noindent\textbf{Step 2:} We now check the induction base step $n=1$. Clearly we only need to check
$$ \sqrt{\frac{2S}{\tau}}+\frac{S}{2}-\frac{\nu}{2}\geq\| \Pi_N u_0\|^2_\infty+\frac12\| u^1\| ^2_\infty.$$
By Lemma \ref{1stLem}, 
\begin{equation*}
\begin{aligned}
\| u^1\|_{\dot{H}^1}&\leq\left(1+\frac1S+\frac1S\| \Pi_N u_0\|^2_\infty\right)\cdot\| u_0\|_{\dot{H}^1}\\
&\leq\left(1+\frac1S+\frac1S\| \Pi_N u_0\|^2_\infty\right)\cdot\sqrt{\frac{2E^0}{\nu}}\ .
\end{aligned}
\end{equation*}
As a result, 
\begin{equation*}
\begin{aligned}
\| u^1\|_\infty&\lesssim1+\| u^1\|_{\dot{H}^1}\sqrt{\log(3+\| u^1\|_{\dot{H}^{\frac32}})}\\
&\lesssim1+\left(1+\frac1S+\frac1S\| \Pi_N u_0\|^2_\infty\right)\sqrt{\frac{2E^0}{\nu}}\sqrt{\log\left(3+c_1\left(\sqrt{\frac{S}{\nu}}+\frac{1}{S}\right)(E_0+1)\right)}
\\
&\lesssim_{E_0}1+\left(1+\frac1S+\frac1S\| \Pi_N u_0\|^2_\infty\right)\cdot\nu^{-\frac12}\cdot\sqrt{1+\log(S)+|\log(\nu)|+\log N}.
\end{aligned}
\end{equation*}
Thus we need to choose $S$ such that
\begin{equation*}
\begin{aligned}
 \sqrt{\frac{2S}{\tau}}+\frac{S}{2}-\frac{\nu}{2}&\geq\| \Pi_N u_0\|^2_\infty+C_{E_0}\cdot\left(1+\frac1S+\frac1S\| \Pi_N u_0\|^2_\infty\right)^2\cdot\nu^{-1}\\&\cdot\left(1+\log(S)+|\log(\nu)|+\log N \right)\ ,
\end{aligned}
\end{equation*}
where $C_{E_0}$ is a large constant depending only on $E_0$. Note that by Morrey's inequality, 
$$\| \Pi_N u_0\|_{L^\infty(\T)}\lesssim\| \Pi_N u_0\|_{H^2(\T)}\lesssim\| u_0\|_{H^2(\T)}.$$
Then it suffices to take $S$ such that
\begin{equation}
S\gg_{E_0}\| u_0\|^2_{H^2}+\nu^{-1}(|\log(\nu)|+\log N)+\nu\ .
\end{equation}
This completes the induction and hence proves the theorem. 

\end{proof}

\section{$L^2$ error estimate for the scheme (\ref{1stScheme}) on the 2D Cahn-Hilliard equation}
\label{section:error}
In this section, we will like to study the $L^2$ error between the semi-implicit numerical solution and the exact PDE solution in the domain $\T$ and eventually prove Theorem~\ref{1st_error}. To start with, we consider the auxiliary $L^2$ error estimate for near solutions.

\subsection{Auxiliary $L^2$ error estimate for near solutions}
Consider the following auxiliary system $u^n$ and $v^n$ in $X_N$ for the first order scheme:
\begin{equation}\label{1st_auxscheme}\left\{
\begin{aligned}
&\frac{u^{n+1}-u^{n}}{\tau}+S\De^2(u^{n+1}-u^n)=\frac{e^{-\nu\De^2\tau}-1}{\tau}u^n+\frac{1-e^{-\nu\De^2\tau}}{\tau\nu\De}\Pi_N f(u^n)+\De G^1_n\\
&\frac{v^{n+1}-v^{n}}{\tau}+S\De^2(v^{n+1}-v^n)=\frac{e^{-\nu\De^2\tau}-1}{\tau}v^n+\frac{1-e^{-\nu\De^2\tau}}{\tau\nu\De}\Pi_N f(v^n)+\De G^2_n\\
&u^0=u_0\ ,\ v^0=v_0\ .
\end{aligned}
\right.
\end{equation}
We define that $G_n=G^1_n-G^2_n$. 
\begin{prop}\label{prop_aux}
For solutions of (\ref{1st_auxscheme}), assume for some $N_1>0$,
\begin{equation}\label{auxCond}
\sup_{n\ge 0}\|\na u^n\|_2+\sup_{n\geq 0}\| v^n\|_{\infty}+\sup_{n\geq 0}\| \na v^n\|_2\leq N_1\  .
\end{equation}
Then for any $m\geq 1$,
\begin{equation}
\begin{aligned}
&\| u^m-v^m\| ^2_2=\| \ve^m\| ^2_2\leq y_m\\
\leq& \exp\left(m\tau\cdot\frac{C(1+N_1^4)}{\nu}\right)\cdot\left(\|L u_0-L v_0\|_2^2+S\tau \|Mu_0-M v_0\|_2^2+\frac{4}{\nu}\tau\sum_{n=0}^{m-1}\| L^2 G_n\|_2^2\right) .
\end{aligned}
\end{equation}
where $C>0$ is an absolute constant and $L, M$ are defined from the Fourier side as below:
\begin{align}
    \begin{cases}
        &\widehat{Lu}(k)=\sqrt{\frac{\nu|k|^4\tau}{1-e^{-\nu|k|^4\tau}}}\widehat{u}(k)\\
        & \widehat{Mu}(k)=\sqrt{\frac{\nu|k|^8\tau}{1-e^{-\nu|k|^4\tau}}}\widehat{u}(k).
    \end{cases}
\end{align}
The operator $L^2=L\circ L$, i.e. $L$ composite with $L$.
\end{prop}
\begin{lemma}\label{lem4.2}
We remark here by similar arguments in Lemma~\ref{lem3.2} we have for $g,h\in C^\infty(\T)$ that the following holds:
\begin{equation}\label{4.16}
    \|\De g\|_2\le \|M g\|_2\lesssim \|\De g\|_2+\sqrt{\tau}\|\De^2 g\|_2,\quad \| h\|_2\le \|L h\|_2\lesssim \|h\|_2+\sqrt{\tau}\|\De h\|_2.
\end{equation}
\end{lemma}
\begin{proof}[Proof of Proposition~\ref{prop_aux}]
We assume Lemma~\ref{lem4.2} for the time being. Write $\ve^n=u^n-v^n$. Then
\begin{equation}\label{4.2}
\frac{\ve^{n+1}-\ve^n}{\tau}+S\De^2(\ve^{n+1}-\ve^{n})=\frac{e^{-\nu\De^2\tau}-1}{\tau}\ve^n+\frac{1-e^{-\nu\De^2\tau}}{\tau\nu\De}\Pi_N (f(u^n)-f(v^n))+\De G_n\ .
\end{equation}
Taking $L^2$-inner product with $\frac{\nu\De^2\tau}{1-e^{-\nu\De^2\tau}}\ve^{n+1}$ on both sides of \eqref{4.2} and recalling similar computations in previous section, we have the LHS of \eqref{4.2} is  
\begin{equation}\label{4.3}
\begin{aligned}
\mbox{LHS}=&\left(\frac{\ve^{n+1}-\ve^n}{\tau}+S\De^2(\ve^{n+1}-\ve^n),\ \frac{\nu\De^2\tau}{1-e^{-\nu\De^2\tau}}\ve^{n+1}\right)\\
=&\frac{1}{2\tau}(\|L\ve^{n+1}\|_2^2-\|L\ve^{n}\|_2^2+\|L(\ve^{n+1}-\ve^{n})\|_2^2)+\frac{S}{2}(\|M\ve^{n+1}\|_2^2-\|M\ve^{n}\|_2^2+\|M(\ve^{n+1}-\ve^{n})\|_2^2).
\end{aligned}
\end{equation}
Similarly, the RHS of \eqref{4.2} can be given as below
\begin{equation}\label{4.5}
\begin{aligned}
\mbox{RHS}=&\left(\frac{e^{-\nu\De^2\tau}-1}{\tau}\ve^n+\De G_n+\frac{1-e^{-\nu\De^2\tau}}{\tau\nu\De}\Pi_N(f(u^n)-f(v^n)),\ \frac{\nu\De^2\tau}{1-e^{-\nu\De^2\tau}}\ve^{n+1}\right)\\
\coloneqq& I_1+I_2+I_3,
\end{aligned}
\end{equation}
where $(.\ ,\ .)$ denotes the $L^2$ inner product.
We then estimate different parts $I_1,I_2$ and $I_3$ as follows.
\begin{equation}\label{I1}
\begin{aligned}
    I_1=&\left(\frac{e^{-\nu\De^2\tau}-1}{\tau}\ve^n,\frac{\nu\De^2\tau}{1-e^{-\nu\De^2\tau}}\ve^{n+1}\right)\\
    =&-\nu(\De \ve^{n},\De \ve^{n+1})\\
    =&-\nu(\De \ve^{n+1},\De \ve^{n+1})+\nu(\De\ve^{n+1},\De\ve^{n+1}-\De\ve^n)\\
    =&-\frac{\nu}{2}(\|\De \ve^{n+1}\|_2^2+\|\De \ve^n\|_2^2)+\frac{\nu}{2}\|\De \ve^{n+1}-\De\ve^{n}\|_2^2.
\end{aligned}
\end{equation}
Similarly, by H\"{o}lder's inequality we obtain that 
\begin{equation}\label{I2}
I_2=\left(\De G_n,\frac{\nu\De^2\tau}{1-e^{-\nu\De^2\tau}}\ve^{n+1}\right)\leq\| \De \ve^{n+1}\|_{2}\| L^2 G_n\|_{2}\le  \frac{2}{\nu} \| L^2 G_n\|^2_{2}+\frac{\nu}{8}\| \De \ve^{n+1}\|_{2}^2,
\end{equation}
where $L^2(G_n)=L(LG_n)$. On the other hand, note that $\Pi_N$ is a self-adjoint operator $\left(\Pi_Nf,\ g\right)=\left(f,\ \Pi_Ng\right)$, since it is just an $N$-th Fourier mode truncation. Therefore we have
\begin{equation}
    \begin{aligned}
        I_3=\left(\frac{1-e^{-\nu\De^2\tau}}{\tau\nu\De}\Pi_N (f(u^n)-f(v^n)),\frac{\nu\De^2\tau}{1-e^{-\nu\De^2\tau}}\ve^{n+1}\right)=\left(f(u^n)-f(v^n),\De \Pi_N \ve^{n+1}\right).
    \end{aligned}
\end{equation}
Then by the fundamental theorem of calculus, we have
\begin{equation}\label{4.10}
\begin{aligned}
f(u^n)-f(v^n)&=\int_0^1f'(v^n+s\ve^n)ds\cdot \ve^n\\
&=(a_1+a_2(v^n)^2)\ve^n+a_3v^n(\ve^n)^2+a_4(\ve^n)^3\ ,
\end{aligned}
\end{equation}
where $a_i$ are exact constants can be computed. Note that we will denote $C$ to be an absolute constant whose value may vary in different lines. To proceed further we indeed have by \eqref{auxCond} that
\begin{equation}\label{4.11}
\begin{aligned}
|\left((a_1+a_2(v^n)^2)\ve^n\ ,\ \De\Pi_N\ve^{n+1}\right)|&\leq C(1+\| v^n\|_\infty^2)\| \De \ve^{n+1}\|_{2}\| \ve^n\|_{2}\\
&\le \frac{C(1+N_1^4)}{\nu}\| \ve^{n}\|_{2}^2+\frac{\nu}{8}\|\De \ve^{n+1}\|_{2}^2.
\end{aligned}
\end{equation}
Moreover, the other two terms can be estimated similarly. In particular, recalling \eqref{2.4} that for $f$ defined in $\T$ with zero mean one has $\|f\|_4^2\lesssim \|f\|_2\|\na f\|_2$. Therefore we have
\begin{equation}\label{4.12}
\begin{aligned}
|\left(a_3v^n(\ve^n)^2\ ,\De \Pi_N\ve^{n+1}\right)|&\leq C\|v^n\|_{\infty}\|\ve^n\|_4^2\|\De \ve^{n+1}\|_2
\le \frac{CN_1^2}{\nu}\|\ve^n\|_4^4+\frac{\nu}{8}\|\De \ve^{n+1}\|_2^2
\\
&\le C\frac{N_1^2}{\nu}\| \ve^{n}\|_2^2\|\na \ve^n\|_2^2+\frac{\nu}{8}  \|\De \ve^{n+1}\|_{2}^2\leq \frac{C N_1^4}{\nu}\| \ve^n\|_{2}^2+\frac{\nu}{8}\|\De \ve^{n+1}\|_2^2.
\end{aligned}
\end{equation}
Similarly, recalling \eqref{2.5} that we also have
\begin{equation}\label{4.13}
\begin{aligned}
|\left(a_4(\ve^n)^3\ ,\De\Pi_N \ve^{n+1}\right)|&\le C\|\ve^{n}\|_6^3\|\De \ve^{n+1}\|_2\le \frac{C}{\nu}\| \ve^{n}\|_6^6 +\frac{\nu}{8}\|\De \ve^{n+1}\|_2^2 \\
&\leq \frac{C}{\nu}\|\ve^n\|_2^2\|\na\ve^n\|_2^4+\frac{\nu}{8}\|\De \ve^{n+1}\|_2^2 \\
&\le\frac{CN_1^4}{\nu}\|\ve^n\|_2^2+\frac{\nu}{8}\|\De \ve^{n+1}\|_2^2.
\end{aligned}
\end{equation}
One can then conclude from \eqref{4.10}, \eqref{4.11}, \eqref{4.12} and \eqref{4.13} that
\begin{equation}\label{I3}
|I_3|=|\left(f(u^n)-f(v^n),\De\Pi_N \ve^{n+1}\right)|\le  \frac{C(1+N_1^4)}{\nu}\|\ve^n\|_2^2+\frac{3\nu}{8}\|\De \ve^{n+1}\|_2^2.
\end{equation}
Collecting the estimates \eqref{4.5}, \eqref{I1}, \eqref{I2} and \eqref{I3}, we get
\begin{equation}\label{4.15}
    \begin{aligned}
        \mbox{RHS}\le \frac{\nu}{2}\|\De\ve^{n+1}-\De\ve^n\|_2^2+\frac{2}{\nu}\|L^2G_n\|_2^2+\frac{C(1+N_1^4)}{\nu}\|\ve^n\|_2^2.
    \end{aligned}
\end{equation} Taking the same choice of $S$ as in Theorem~\ref{Thm_1stab}, it then follows by \eqref{4.3} and \eqref{4.15} that
\begin{equation}\label{4.17}
\begin{aligned}
&\frac{\| L\ve^{n+1}\| ^2_2-\| L\ve^n\|_2^2}{\tau}+S(\| M\ve^{n+1}\| ^2_2-\|M \ve^n\|_2^2)
\\ \leq &\frac{4}{\nu}\| L^2G_n\|_2^2+\frac{C(1+N_1^4)}{\nu}\|\ve^n\|_2^2,
\end{aligned}
\end{equation}
where $C$ is an absolute positive constant that can be computed exactly. Then by defining  
\begin{equation}
\begin{aligned}
&y_n=\|L\ve^n\|_2^2+\tau S\|M\ve^n\|_2^2
,\\ &\alpha= \frac{C(1+N_1^4)}{\nu},
\\ &\beta_n=\frac{4}{\nu}\|L^2 G_n\|_2^2,
\end{aligned}
\end{equation}
\eqref{4.17} thus can be formulated into
$$ \frac{y_{n+1}-y_n}{\tau}\leq \alpha y_n+\beta_n\ . $$
Applying discrete Gr\"onwall's inequality Lemma \ref{DiscreteGronwall}, we arrive at
\begin{equation}
\begin{aligned}
&\| u^m-v^m\| ^2_2=\| \ve^m\| ^2_2\leq y_m\\
\leq& \exp\left(m\tau\cdot\frac{C(1+N_1^4)}{\nu}\right)\cdot\left(\|L u_0-L v_0\|_2^2+S\tau \|Mu_0-M v_0\|_2^2+\frac{4}{\nu}\tau\sum_{n=0}^{m-1}\| L^2 G_n\|_2^2\right) .
\end{aligned}
\end{equation}

\end{proof}
\begin{remark}
    The proof of Proposition~\ref{prop_aux} only requires the uniform boundedness of $\|\na u^n\|_2+\|\na v^n\|_2$ rather than $\|\De u^n\|_2+\|\De v^n\|_2$. It somehow makes sense by considering the equation \eqref{1stScheme} where the nonlinear part requires lower regularity assumption and the linear part can be controlled by its sign.
\end{remark}

\begin{proof}[Proof of Lemma~\ref{lem4.2}]
 First of all the same arguments from Lemma~\ref{lem3.2} give
 \begin{align*}
     \|\De g\|_2\le \|Mg\|_2,\quad \|h\|_2\le\|L h\|_2.
 \end{align*}
To see the rest, we consider the Fourier series of $Lh$ again:
\begin{equation}\label{4.21}
   \begin{aligned}
       \|Lh\|_2^2=&\frac{1}{(2\pi)^2}\sum_{|k|\ge 1}\frac{\nu|k|^4\tau}{1-e^{-\nu|k|^4\tau}}|\widehat{h}(k)|^2\\
       =&\frac{1}{(2\pi)^2}(\sum_{\nu|k|^4\tau\le \log 2}+\sum_{\nu|k|^4\tau\ge \log 2}\frac{\nu|k|^4\tau}{1-e^{-\nu|k|^4\tau}}|\widehat{h}(k)|^2 )\\
       \lesssim&\sum_{\nu|k|^4\tau\le \log 2}\frac{\nu|k|^4\tau}{1-e^{-\nu|k|^4\tau}}|\widehat{h}(k)|^2 +\tau \sum_{\nu|k|^4\tau\ge \log 2}|k|^4|\widehat{h}(k)|^2.
   \end{aligned} 
\end{equation}
The high frequency part can be controlled by $\|\De h\|_2$ by using the standard Parseval's
identity. For the lower frequency part, we then consider the function $\phi(x)=\frac{x}{1-e^{-x}}$. By direct computation, we see that $\phi'(x)>0$ for $x>0$ which shows that $\phi(x)$ is monotone increasing. Therefore we have
\begin{equation}\label{4.22}
    \begin{aligned}
        \sum_{\nu|k|^4\tau \le \log 2}\frac{\nu |k|^4\tau}{1-e^{-\nu|k|^4\tau}} |\widehat{h}(k)|^2\lesssim \sum_{|k|\ge 1}|\widehat{h}(k)|^2.
    \end{aligned}
\end{equation}
Collecting the estimates \eqref{4.21} and \eqref{4.22} we then get the desired result. The same reason gives 
\begin{align*}
    \|Mg\|_2\lesssim\|\De g\|_2+\sqrt{\tau}\|\De^2 g\|_2.
\end{align*}
\end{proof}



\subsection{$L^2$ error estimate of the 2D Cahn-Hilliard equation}
In this section, to simplify the notation, we will write $x\lesssim y$ if $x\leq C(\nu, u_0)\ y$ for a constant $C$ depending on $\nu$ and $u_0$. 
We consider the system
\begin{equation}\label{1st_contscheme}
\left\{
\begin{aligned}
&\frac{u^{n+1}-u^{n}}{\tau}+S\De^2(u^{n+1}-u^n)=\frac{e^{-\nu\De^2\tau}-1}{\tau}u^n+\frac{1-e^{-\nu\De^2\tau}}{\tau\nu\De}\Pi_N f(u^n)\\
&\partial_tu=-\nu\Delta^2 u+\De( f(u))\\
&u^0=\Pi_Nu_0\ ,\ u(0)=u_0\ .
\end{aligned}
\right.
\end{equation}
 In order to prove Theorem~\ref{1st_error}, it is clear that we shall estimate $G_n$ introduced in \eqref{1st_auxscheme} from the previous proposition. Note that for a one-variable function $h(t)$, one has the formula:

\begin{equation}\label{hform}
\left\{
\begin{aligned}
&\frac{1}{\tau}\int_{t_n}^{t_{n+1}}h(t)=h(t_n)+\frac{1}{\tau}\int_{t_n}^{t_{n+1}}h'(t)\cdot(t_{n+1}-t)\ dt\\
&\frac{1}{\tau}\int_{t_n}^{t_{n+1}}h(t)=h(t_{n+1})+\frac{1}{\tau}\int_{t_n}^{t_{n+1}}h'(t)\cdot(t_n-t)\ dt\ .
\end{aligned}
\right.
\end{equation}
Using the formula \eqref{hform} above and integrating the Cahn-Hilliard equation \eqref{CH} on the time interval $[t_n , t_{n+1}]$, we get
\begin{equation}\label{4.23}
\begin{aligned}
\frac{u(t_{n+1})-u(t_n)}{\tau}=
&-\nu\Delta^2 u(t_{n})-S\De^2\left( u(t_{n+1})-u(t_n)\right)\\&+\De\Pi_Nf(u)(t_n)+\De\Pi_{>N}f(u)(t_n)+\De Q_n\ ,
\end{aligned}
\end{equation}
where $\Pi_{>N}=id-\Pi_N$, the large mode truncation operator, and 
\begin{equation}
\begin{aligned}
Q_n=-\frac{\nu}{\tau}\int_{t_n}^{t_{n+1}} \partial_t\Delta u\cdot(t_{n+1}-t)\ dt+\frac{1}{\tau}\int_{t_n}^{t_{n+1}}\partial_t(f(u))(t_{n+1}-t)\ dt+S\int_{t_n}^{t_{n+1}}\partial_t\De u\ dt\ .
\end{aligned}
\end{equation}
We then rewrite \eqref{4.23} into the following form:
\begin{equation}\label{4.25}
\begin{aligned}
\frac{u(t_{n+1})-u(t_n)}{\tau}+&S\De^2\left( u(t_{n+1})-u(t_n)\right)=\frac{e^{-\nu\De^2\tau}-1}{\tau}u(t_n)+\frac{1-e^{-\nu\De^2\tau}}{\tau\nu\De}\Pi_N f(u(t_n))\\
&+\left(\frac{1-e^{-\nu\De^2\tau}}{\tau}-\nu\Delta^2\right) u(t_{n})
+\left(\De + \frac{e^{-\nu\De^2\tau}-1}{\nu\tau\De}\right)\Pi_N f(u(t_n))
\\&+\De\Pi_{>N}f(u)(t_n)+\De Q_n\ ,
\end{aligned}
\end{equation}
For the sake of simplicity, we denote $H_1, H_2, H_3$ to be the following quantities:
\begin{equation}
\begin{aligned}
    \begin{cases}
        &H_1\coloneqq \left(\frac{1-e^{-\nu\De^2\tau}}{\tau\De}-\nu\Delta \right) u(t_{n})=\frac{1-\nu\De^2\tau-e^{-\nu\De^2\tau}}{\tau\De} u(t_n),\\
       & H_2\coloneqq \left(1+\frac{e^{-\nu\De^2\tau}-1}{\nu\tau\De^2}\right)\Pi_N f(u)(t_n),\\
       &H_3\coloneqq \Pi_{>N}f(u)(t_n).
    \end{cases}
\end{aligned}
\end{equation}
In order to prove Theorem~\ref{1st_error}, we now estimate $\|L^2 G_n\|_2$ by estimating $\|L^2Q_n\|_2$ and $\|L^2 H_i\|_2$, $i=1,2,3$. Recall that $L^2u=\frac{\nu\De^2\tau}{1-e^{-\nu\De^2\tau}}u$.
\begin{proof}[Proof of Theorem~\ref{1st_error}]
By Theorem~\ref{Thm_1stab} and Lemma \ref{H^kregularity_CH}, $\sup_{n\geq 0}\| u^n\|_{H^1}\lesssim 1$ and $\sup_{t>0}\| u(t)\|_{H^s}\lesssim 1$ for $s\ge 10$ given in the statement. Thus the assumptions of Proposition \ref{prop_aux} (auxiliary $L^2$ error estimate proposition) are satisfied for \eqref{1st_contscheme} and \eqref{4.25}. Recall that
\begin{equation*}
\begin{aligned}
Q_n=-\frac{\nu}{\tau}\int_{t_n}^{t_{n+1}} \partial_t\Delta u\cdot(t_{n+1}-t)\ dt+\frac{1}{\tau}\int_{t_n}^{t_{n+1}}\partial_t(f(u))(t_{n+1}-t)\ dt+S\int_{t_n}^{t_{n+1}}\partial_t\De u\ dt\ .
\end{aligned}
\end{equation*}
Recalling \eqref{4.16} that we can control $\|L^2Q_n\|_2\lesssim \|Q_n\|_2+\sqrt{\tau}\|\De Q_n\|_2+\tau\|\De^2Q_n\|_2$ as follows:
\begin{equation}\label{eq:G_nest}
\begin{aligned}
\|Q_n\|_2&\lesssim \int_{t_n}^{t_{n+1}} \| \partial_t\Delta u\|_2\ dt+\int_{t_n}^{t_{n+1}}\| \partial_t(f(u))\|_2\ dt+S\int_{t_n}^{t_{n+1}}\| \partial_t\De u\|_2\ dt\\
&\lesssim(1+S)\int_{t_n}^{t_{n+1}} \| \partial_t\Delta u\|_2\ dt+\int_{t_n}^{t_{n+1}}\| \partial_t  u\|_2\ dt\cdot \| f'(u)\|_{L^\infty_tL^\infty_x}.
\end{aligned}
\end{equation}
Note that by Lemma \ref{H^kregularity_CH}, 
\begin{equation*}
\| \partial_t u\|_2+\|\pa_t\De u\|_2+ \| f'(u)\|_\infty\lesssim 1\ .
\end{equation*}
Therefore we can estimate \eqref{eq:G_nest} as following
\begin{equation*}
\begin{aligned}
\int_{t_n}^{t_{n+1}} \| \partial_t\Delta u\|_2\ dt\lesssim\left(\int_{t_n}^{t_{n+1}} \| \partial_t\Delta u\|_2^2\ dt\right)^{\frac12}\cdot\sqrt{\tau}\ .
\end{aligned}
\end{equation*}
Therefore when $t_m\geq 1$, it is not hard to see that
\begin{equation}
\begin{aligned}
\sum_{n=0}^{m-1}\| Q_n\|_2^2
&\lesssim\sum_{n=0}^{m-1}\left(\tau(1+S)^2\int_{t_n}^{t_{n+1}}\| \partial_t\Delta u\|_2^2\ dt+\tau\int_{t_n}^{t_{n+1}}\| \partial_t  u\|_2^2\ dt\right)\\
&\lesssim\tau(1+S)^2\int_0^{t_m}\| \partial_t\Delta u\|_2^2\ dt+\tau\int_0^{t_m} \| \partial_t  u\|_2^2\ dt\\
&\lesssim \tau(1+S)^2(1+t_m)+\tau\\
&\lesssim(1+S)^2\tau\cdot(1+t_m)\ .
\end{aligned}
\end{equation}
For the rest terms $\sqrt{\tau}\|\De Q_n\|_2 $ and $\tau \|\De^2 Q_n\|_2$ we will obtain very similar results as long as $s\ge 10$ and thereby we can conclude that:
\begin{align}\label{LG}
   \sum_{n=0}^{m-1} \|L^2Q_n\|_2^2\lesssim (1+S)^2\tau\cdot (1+t_m).
\end{align}
We now focus on the $\|L^2H_i\|_2$ terms. Firstly it is easy to see from the Fourier side that
\begin{equation}\label{H1}
\begin{aligned}
    \|L^2H_1 \|_2^2=&\frac{1}{(2\pi)^2}\sum_{|k|\ge 1}\frac{\nu^2|k|^4(1-\nu|k|^4\tau-e^{-\nu|k|^4\tau})^2}{(1-e^{-\nu|k|^4\tau})^2}|\widehat{u}(k)|^2(t_n)\\
    \lesssim&\sum_{\nu|k|^4\tau\le c }+\sum_{\nu|k|^4\tau\ge c}\frac{|k|^4(1-\nu|k|^4\tau-e^{-\nu|k|^4\tau})^2}{(1-e^{-\nu|k|^4\tau})^2}|\widehat{u}(k)|^2(t_n)\\
    \coloneqq& J_1+J_2.
\end{aligned}
\end{equation}
For the part where $\nu|k|^4\tau\ge c$, we indeed have
\begin{equation}\label{J1}
\begin{aligned}
 J_2\coloneqq&\sum_{\nu|k|^4\tau\ge c}\frac{|k|^4(1-\nu|k|^4\tau-e^{-\nu|k|^4\tau})^2}{(1-e^{-\nu|k|^4\tau})^2}|\widehat{u}(k)|^2(t_n)\\
  \lesssim&\sum_{\nu|k|^4\tau\ge c}|k|^{12}\tau^2|\widehat{u}(k)|^2(t_n)\\
  \lesssim&\tau^2 \|\De^3 u\|^2_2(t_n).
\end{aligned}
\end{equation}
For the lower frequency part, we get
\begin{align*}
 J_1\coloneqq &\sum_{\nu|k|^4\tau\le c}\frac{|k|^4(1-\nu|k|^4\tau-e^{-\nu|k|^4\tau})^2}{(1-e^{-\nu|k|^4\tau})^2}|\widehat{u}(k)|^2(t_n)\\
\lesssim&\sum_{1\le|k|\lesssim(\frac{1}{\tau})^{\frac14}}\frac{|k|^8(1-\nu|k|^4\tau-e^{-\nu|k|^4\tau})^2}{(1-e^{-\nu|k|^4\tau})^2}|\widehat{u}(k)|^2(t_n).
\end{align*}
Recall that $\phi(x)=\frac{x}{1-e^{-x}}$ is monotone increasing for $x>0$ as in Lemma~\ref{lem4.2}. Moreover, we have (by a simple observation) that
\begin{align}\label{4.34}
    0<\nu|k|^4\tau+e^{-\nu|k|^4\tau}-1<\nu^2|k|^8\tau^2.
\end{align}
Therefore we have
\begin{equation}\label{J2}
\begin{aligned}
    J_1\lesssim&\frac{1}{\tau^2}\sum_{1\le|k|\lesssim(\frac{1}{\tau})^{\frac14}}(\nu^2|k|^8\tau^2)^2|\widehat{u}(k)|^2(t_n)\lesssim\tau^2\sum_{|k|\ge 1}|k|^{16}|\widehat{u}(k)|^2(t_n)\lesssim\tau^2\|\De^4u\|_2^2(t_n).
\end{aligned}
\end{equation}
Combing \eqref{J1}-\eqref{J2}, we get
\begin{align}\label{LH1}
  \sum_{n=0}^{m-1}  \|L^2H_1\|_2^2\lesssim \sum_{n=0}^{m-1}\tau^2(\|\De^3u\|_2^2(t_n)+\|\De^4u\|_2^2(t_n))\lesssim m\tau^2\lesssim (t_m+1)\cdot \tau,
\end{align}
where the uniform Sobolev bound lemma (Lemma \ref{H^kregularity_CH}) is applied.

Secondly we estimate $\|L^2H_2\|_2^2$. Note that again by considering the Fourier expansion we arrive at:
\begin{equation}
    \begin{aligned}
        \|L^2H_2\|_2^2=&\frac{1}{(2\pi)^2}\sum_{1\le|k|\le N }\frac{(\nu\tau|k|^4+e^{-\nu|k|^4\tau}-1)^2 }{(1-e^{-\nu\tau|k|^4})^2}|\widehat{f(u)}(k)|^2(t_n)\\
        \lesssim&\sum_{\nu|k|^4\tau\le c}+\sum_{\nu|k|^4\tau\ge c}\frac{(\nu\tau|k|^4+e^{-\nu|k|^4\tau}-1)^2 }{(1-e^{-\nu\tau|k|^4})^2}|\widehat{f(u)}(k)|^2(t_n)\\
        \coloneqq&J_3+J_4.
    \end{aligned}
\end{equation}
Similar to \eqref{H1} above, we have
\begin{equation}
    \begin{aligned}
        J_4\coloneqq&\sum_{\nu|k|^4\tau\ge c}\frac{(\nu\tau|k|^4+e^{-\nu|k|^4\tau}-1)^2 }{(1-e^{-\nu\tau|k|^4})^2}|\widehat{f(u)}(k)|^2(t_n)\\
\lesssim&\sum_{|k|\ge1}|k|^8\tau^2|\widehat{f(u)}(k)|^2(t_n)\\
        \lesssim&\tau^2\|\De^2f(u)\|_2^2(t_n)\lesssim \tau^2.
    \end{aligned}
\end{equation}
Note that here by Lemma~\ref{H^kregularity_CH} we have
$\sup_{t\geq 0}\| u\|_{H^s}(t)\lesssim_s 1$, which leads to $\sup_{n\geq 0}\| f(u)\|_{H^s}(t_n)\lesssim_s 1$. For $J_3$, we have
\begin{equation}
    \begin{aligned}
        J_3\coloneqq&\sum_{\nu|k|^4\tau\le c}\frac{(\nu\tau|k|^4+e^{-\nu|k|^4\tau}-1)^2 }{(1-e^{-\nu\tau|k|^4})^2}|\widehat{f(u)}(k)|^2(t_n)\\
        \lesssim&\sum_{1\le|k|\lesssim(\frac{1}{\tau})^{\frac14}}\frac{|k|^8(\nu|k|^4\tau+e^{-\nu|k|^4\tau}-1)^2}{(1-e^{-\nu|k|^4\tau})^2}|\widehat{f(u)}(k)|^2(t_n)\\
        \lesssim&\frac{1}{\tau^2}\sum_{1\le|k|\lesssim(\frac{1}{\tau})^{\frac14}}\tau^4|k|^{16}|\widehat{f(u)}(k)|^2(t_n)\lesssim\tau^2\|\De^4f(u)\|_2^2(t_n),
    \end{aligned}
\end{equation}
where again we use the monotonicity of $\phi(x)=\frac{x}{1-e^{-x}}$ for $x>0$ and \eqref{4.34}. We can then derive that
\begin{align}\label{LH2}
    \sum_{n=0}^{m-1}\|L^2H_2\|_2^2\lesssim m\tau^2\lesssim(t_m+1)\cdot \tau.
\end{align}
Lastly let us focus on $\|L^2H_3\|_2^2$. Note that $\|L^2H_3\|_2^2\lesssim \|H_3\|_2^2+\tau\|\De H_3\|_2^2+\tau^2\|\De^2 H_3\|_2^2$, we then estimate $\|H_3\|_2^2$ below.
\begin{equation}\label{4.41}
\begin{aligned}
\|H_3\|_2^2= \| \Pi_{>N}f(u)\|_2^2(t_n)&=\sum_{|k|>N}\left|\widehat{f(u)}(k)\right|^2(t_n)\\
&\leq \sum_{|k|>N}|k|^{2s}\left|\widehat{f(u)}(k)\right|^2(t_n)\cdot{|k|^{-2s}}\\
&\lesssim N^{-2s}\cdot \| f(u)\|_{H^s}^2(t_n)\\
&\lesssim N^{-2s}\ .
\end{aligned}
\end{equation}
Similarly, we get $\tau\|\De H_3\|_2^2\lesssim \tau N^{-2s}$ and $\tau^2\|\De^2H_3\|_2^2\lesssim \tau^2 N^{-2s}$. As a result, assuming $\tau\lesssim 1$ or the time step is bounded (which is very natural) we get
\begin{equation}\label{LH3}
\sum_{n=0}^{m-1}\| \Pi_{>N}f(u(t_n))\|_2^2\lesssim_s m\cdot (1+\tau^2)N^{-2s}\lesssim \frac{t_mN^{-2s}}{\tau}\ .
\end{equation}
Therefore, combining \eqref{LG}, \eqref{LH1}, \eqref{LH2} and \eqref{LH3} we can conclude that
\begin{equation}\label{4.43}
\tau\sum_{n=0}^{m-1}\left(\|L^2 Q_n\|_2^2+\sum_{i=1,2,3}\| L^2 H_i\|_2^2\right)\lesssim_s(1+t_m)(\tau^2+N^{-2s})(1+S)^2\ .
\end{equation}
In order to apply Proposition~\ref{prop_aux}, it remains to control $\|Lu^0-Lu(0))\|_2$ and $\|Mu^0-Mu(0)\|_2$.
Note that by \eqref{4.16} in Lemma~\ref{lem4.2} and similar arguments in \eqref{4.41} we have
\begin{equation}\label{4.44}
    \begin{aligned}
        \|Lu^0-Lu(0)\|_2^2=\|L\Pi_Nu_0-Lu_0\|_2^2\lesssim \|\Pi_Nu_0-u_0\|_2^2+\tau\|\De(\Pi_Nu_0-u_0)\|_2^2\lesssim N^{-2s}+\tau N^{-2(s-2)}.
    \end{aligned}
\end{equation}
Similarly, we have
\begin{equation}\label{4.45}
\| Mu^0-Mu(0)\|_2^2=\| M\Pi_Nu_0-Mu_0\|_2^2\lesssim N^{-2(s-2)}+\tau N^{-2(s-4)}\ .
\end{equation}
We now apply the auxiliary solutions estimate in Proposition \ref{prop_aux}. Noting the estimates \eqref{4.43}, \eqref{4.44} and \eqref{4.45} we can get
\begin{equation*}
\| u^m-u(t_m)\|_2^2\lesssim_s(1+S)^2e^{Ct_m}\left(N^{-2s}+\tau^2\cdot N^{-2(s-4)}+\tau^2\right)\ .
\end{equation*}
Thus
\begin{equation}
\| u^m-u(t_m)\|_2\leq (1+S)\cdot C_2\cdot e^{C_1t_m}\left(N^{-s}+\tau+\tau\cdot N^{-s+4}\right)\ ,
\end{equation}
where $C_1>0$ is a constant depending on $\nu$ and $u_0$; $C_2>0$ is a constant depending on $s\ ,\nu$ and $u_0$. This completes the proof of $L^2$ error estimate.
\end{proof}

\section{Numerical experiments}
In this section, we present several numerical results. To begin with we present numerical evidence that show the necessity of the stabilizers.
\subsection{A benchmark computation with different $S$-values}
 In the first several experiments we vary the choice of the stabilizers, namely put $S=0,0.01,0.1$. Moreover for the rest of the parameters, we choose $\nu = 0.01$, $\tau= 0.1,~N_x=N_y = 256$ and the initial data $u_0$ is given basically ``supported'' in seven circles as below:
\begin{align}\label{5.1}
    u_0(x,y)=-1+\sum_{i=1}^7 f_0\left(\sqrt{(x-x_i)^2+(y-y_i)^2}-r_i\right),
\end{align}
where
\begin{align*}
    f_0(s)=\begin{cases}
        &2e^{-\frac{\nu}{s^2}},\quad \mbox{if}\ s<0;\\
        &0, \quad \mbox{otherwise}. 
    \end{cases}
\end{align*}
The centers and radii of the chosen circles are given in the Table~\ref{table:1} below.
\begin{table}[h!]
\centering
\begin{tabular}{||c ||c c c c c c c||} 
 \hline
 $x_i$ & $-\frac{\pi}{2}$ & $-\frac{3\pi}{4}$ &$-\frac{\pi}{2}$ &0& $\frac{\pi}{2}$&0&$\frac{\pi}{2}$ \\ [0.5ex]
 $y_i$ &$-\frac{\pi}{2}$& $-\frac{\pi}{4}$ & $\frac{\pi}{4}$&$-\frac{3\pi}{4}$ & $-\frac{3\pi}{4}$& 0 &$\frac{\pi}{2}$\\[0.5ex]
 $r_i$ & $\frac{\pi}{5}$ &$\frac{2\pi}{15}$ & $\frac{2\pi}{15}$& $\frac{\pi}{10}$&$\frac{\pi}{10}$&$\frac{\pi}{4}$&$\frac{\pi}{4}$ \\[0.5ex]
 \hline 
\end{tabular}
\caption{Choice of centers and radii}
\label{table:1}
\end{table}

The cases $S=0$ (left) and $S=0.01$ (right) can be found in Fig~\ref{fig1} and the case $S=0.1$ can be found in the left of Fig~\ref{fig3}. As you can tell, the choice of $S=0.1$ can guarantee the energy dissipation already. Recall the discussion in Theorem~\ref{Thm_1stab} and Remark~\ref{rem1.6}, we see that the optimal size of $S$ could be much smaller than $\nu^{-1}\log \nu$ that we will not dig further in this direction.

\begin{figure}[htb]

\begin{minipage}[t]{0.48\textwidth}
\centering
\includegraphics[width=0.45\textwidth]{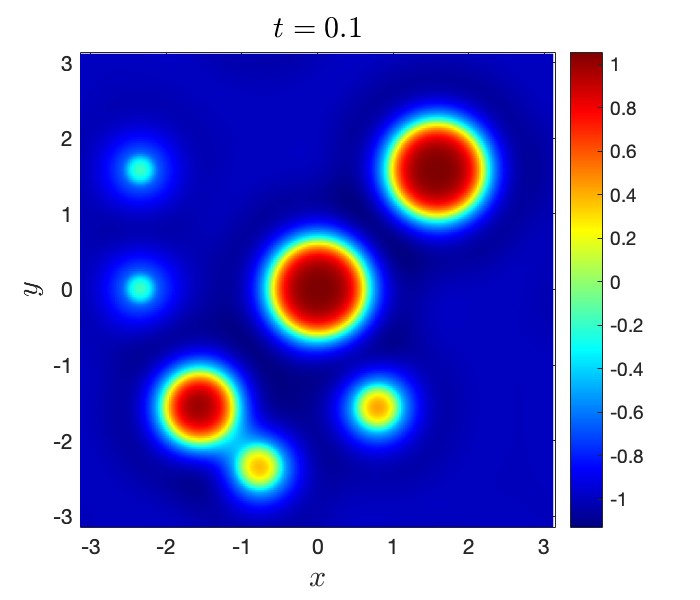}
\includegraphics[width=0.45\textwidth]{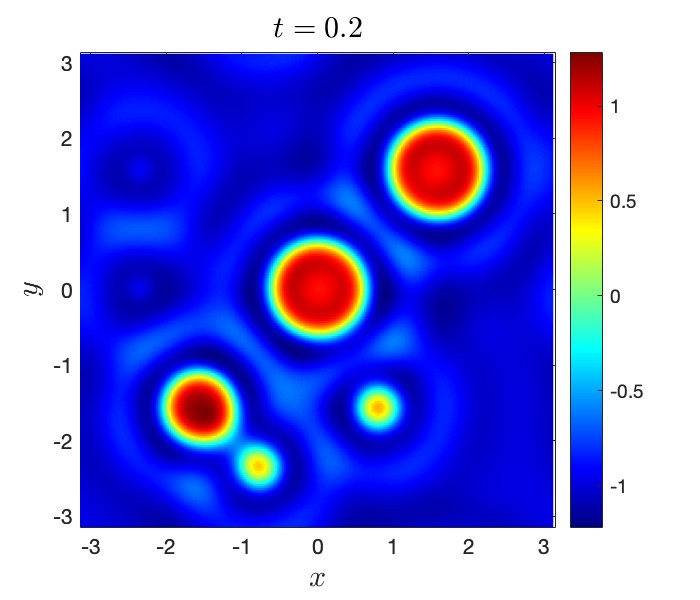}\\
\includegraphics[width=0.45\textwidth]{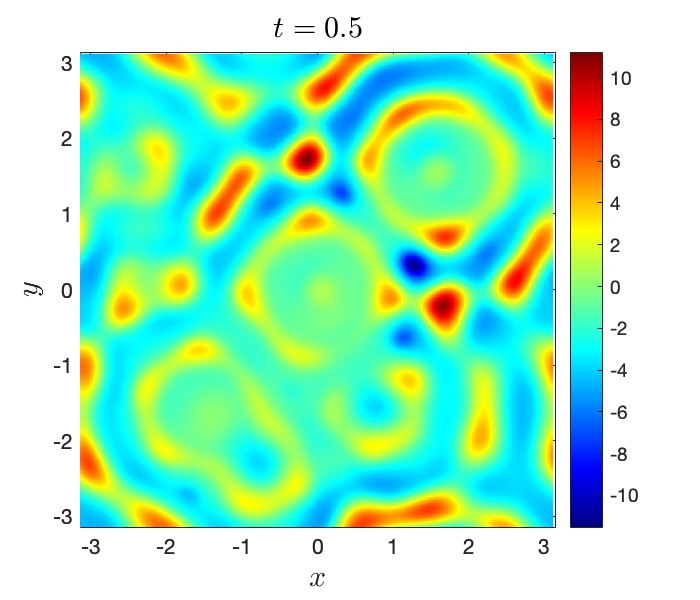}
\includegraphics[width=0.45\textwidth]{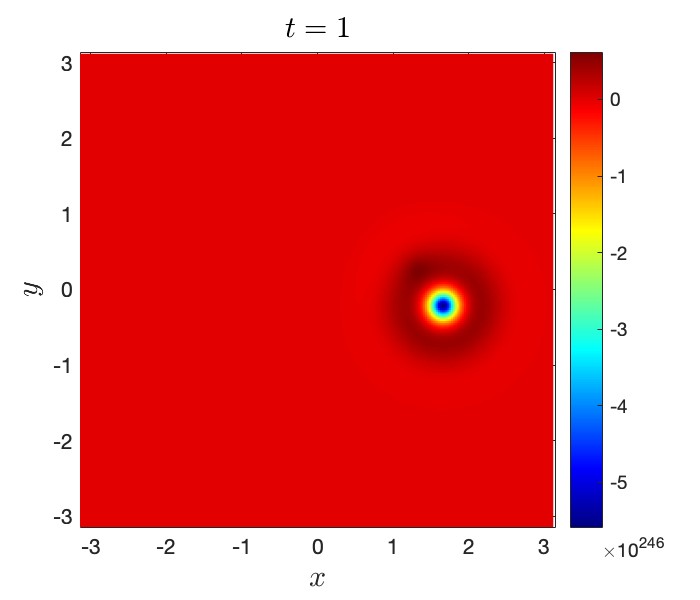}
\includegraphics[width=0.60\textwidth]{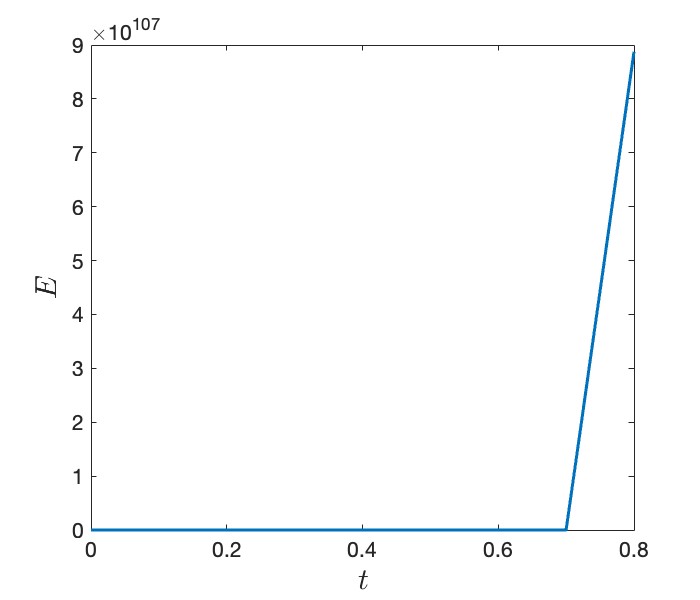}
\end{minipage}
\hfill
\begin{minipage}[t]{0.48\textwidth}
\centering
\includegraphics[width=0.45\textwidth]{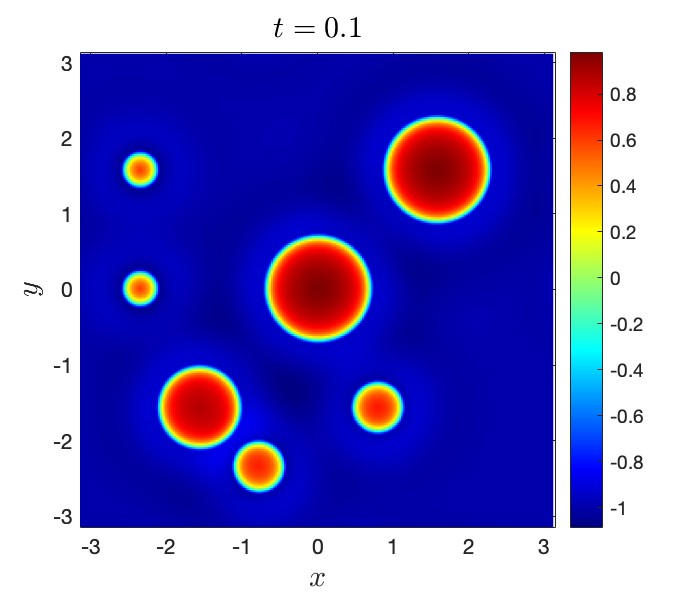}
\includegraphics[width=0.45\textwidth]{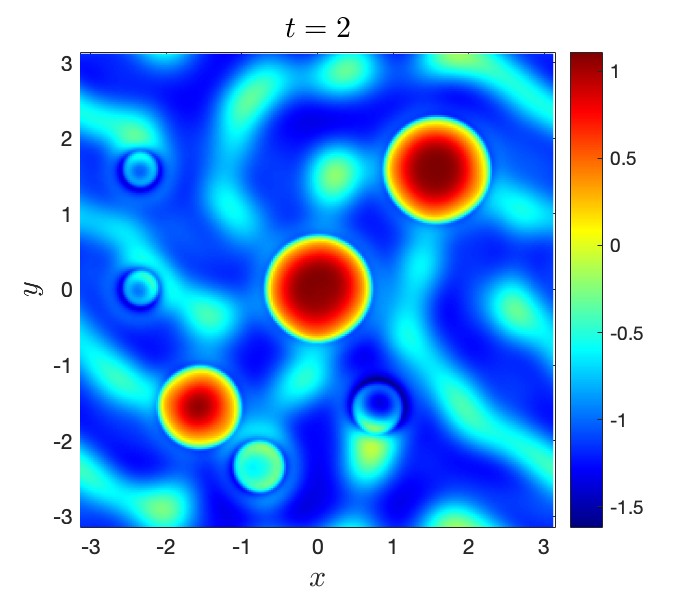}\\
\includegraphics[width=0.45\textwidth]{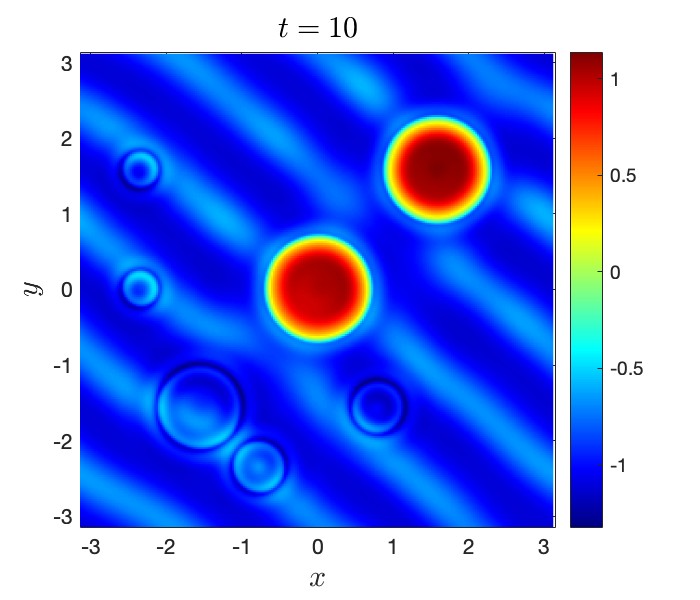}
\includegraphics[width=0.45\textwidth]{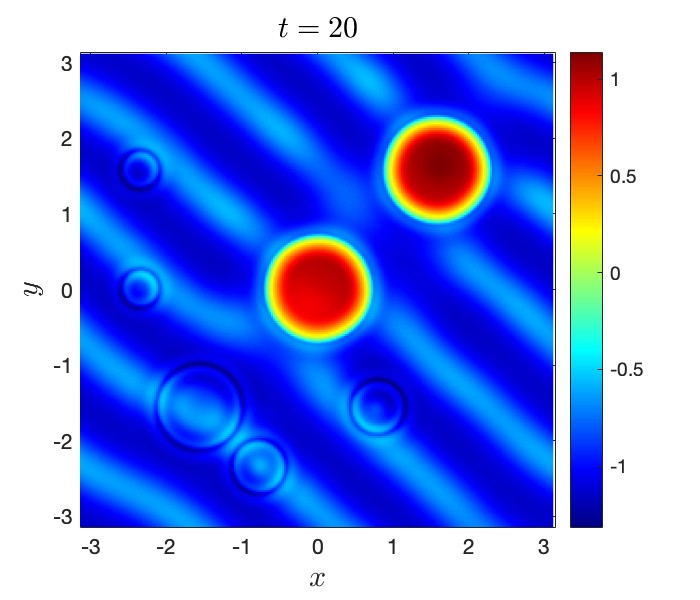}
\includegraphics[width=0.60\textwidth]{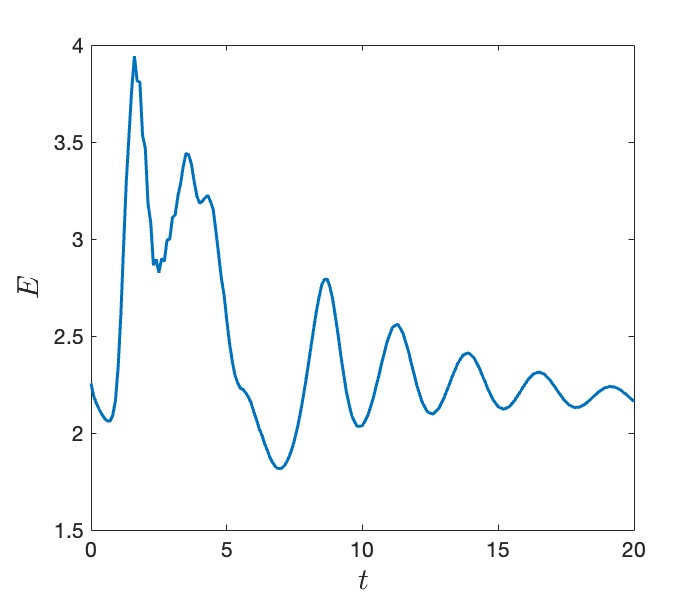}
\end{minipage}

\caption{\small Dynamics of 2D Cahn-Hilliard equation using EI scheme \eqref{1stScheme} where $\nu = 0.01$, $\tau= 0.1,~N_x=N_y = 256$ and the initial data $u_0$ is given in \eqref{5.1}. $S=0$ (Left), $S=0.01$ (right).}\label{fig1}
\end{figure}

\begin{figure}[!h]
\begin{minipage}[t]{0.48\textwidth}
\centering
\includegraphics[width=0.45\textwidth]{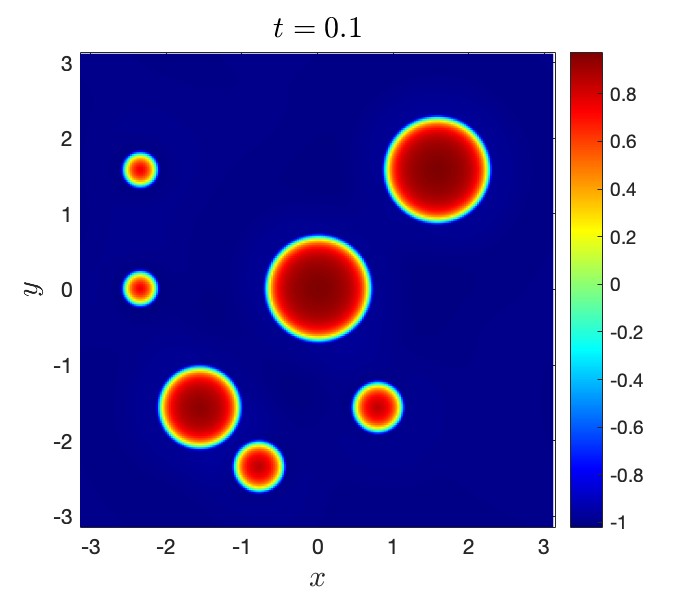}
\includegraphics[width=0.45\textwidth]{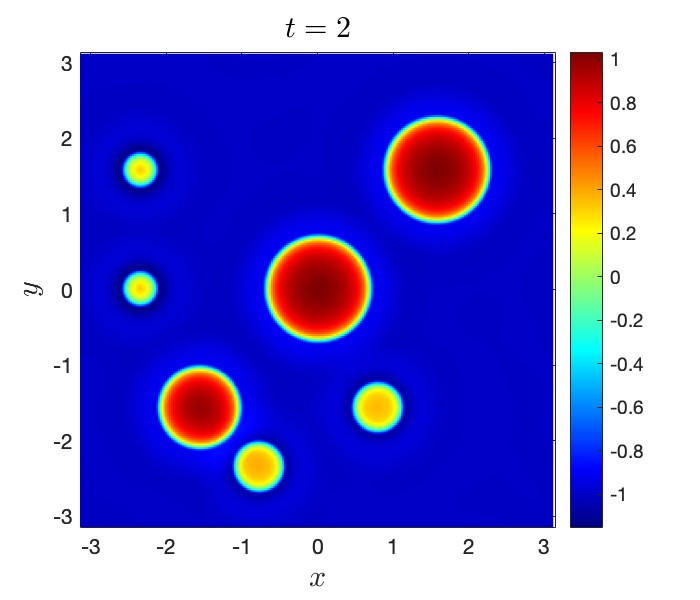}\\
\includegraphics[width=0.45\textwidth]{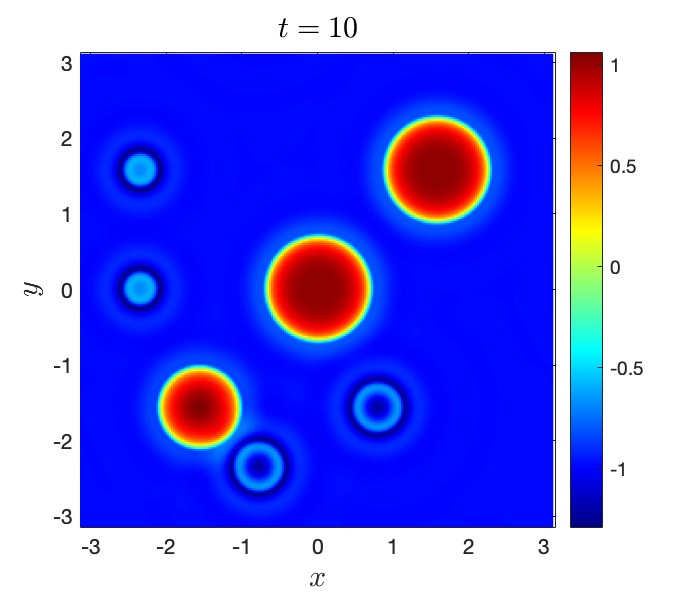}
\includegraphics[width=0.45\textwidth]{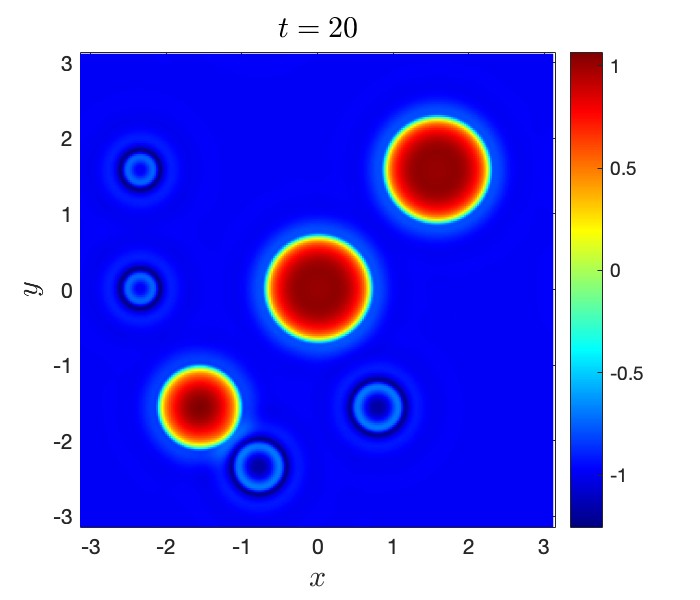}
\includegraphics[width=0.60\textwidth]{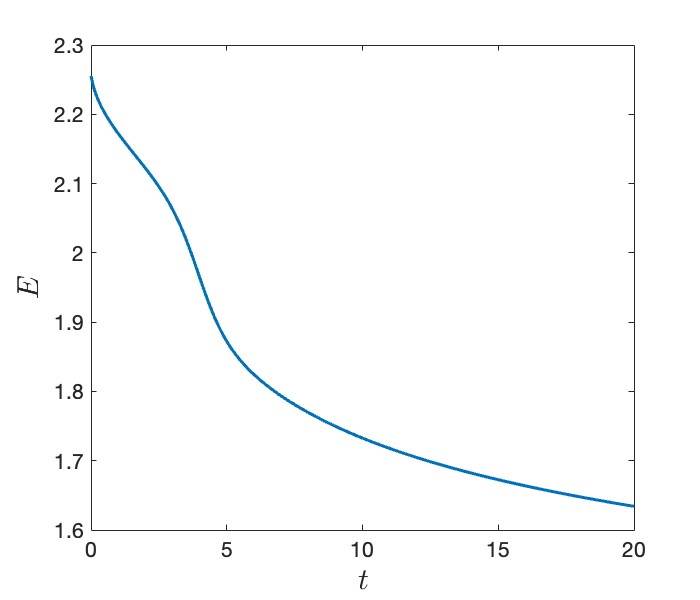}
\caption{\small Dynamics of 2D Cahn-Hilliard equation using EI scheme \eqref{1stScheme} where $\nu = 0.01$, $S=0.1$, $\tau= 0.1,~N_x=N_y = 256$ and the initial data $u_0$ is given in \eqref{5.1}. }\label{fig3}
\end{minipage}
\hfill
\begin{minipage}[t]{0.48\textwidth}
\centering
\includegraphics[width=0.45\textwidth]{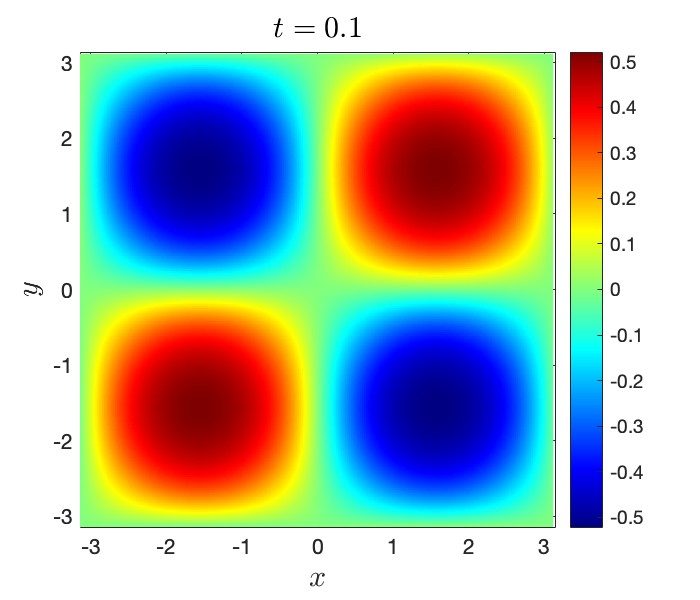}
\includegraphics[width=0.45\textwidth]{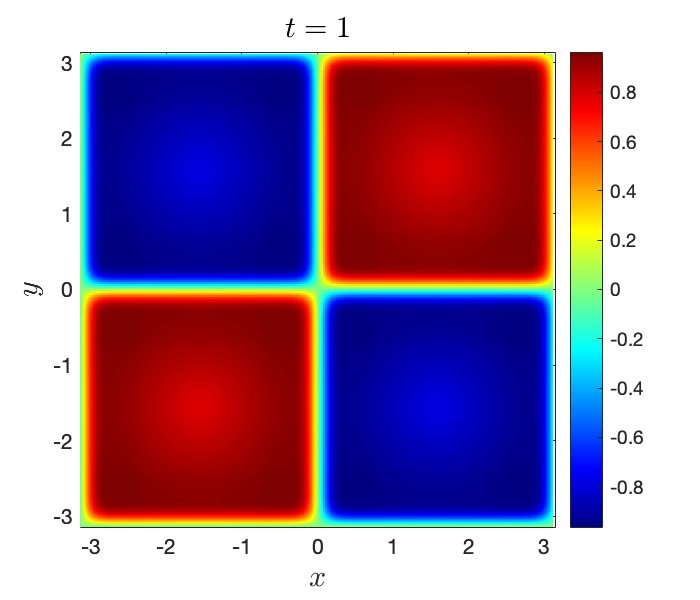}\\
\includegraphics[width=0.45\textwidth]{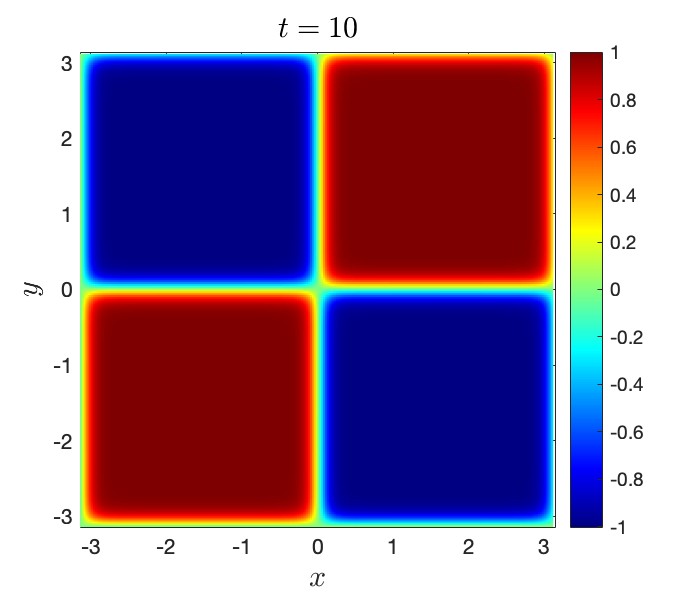}
\includegraphics[width=0.45\textwidth]{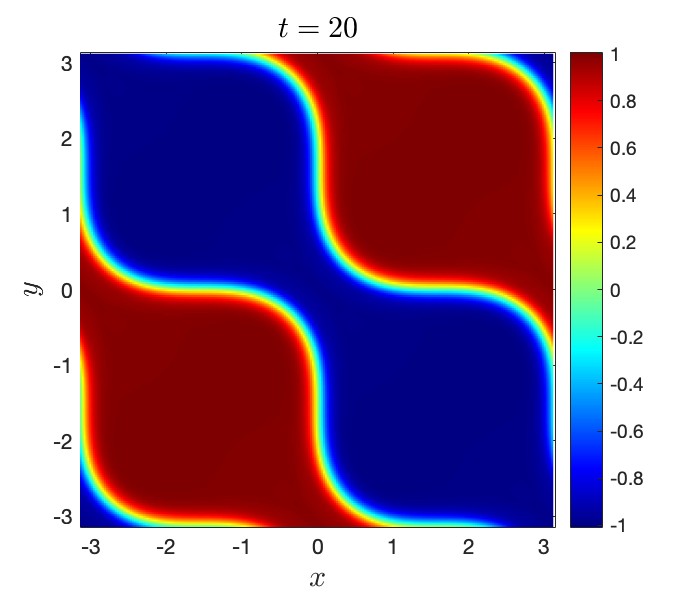}
\includegraphics[width=0.60\textwidth]{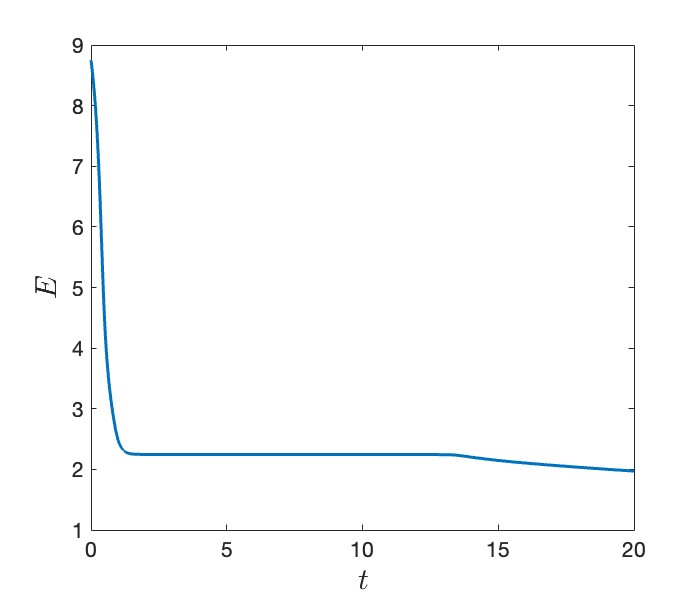}
\caption{\small Dynamics of 2D Cahn-Hilliard equation using EI scheme \eqref{1stScheme} where $\nu = 0.01$, $S=0.1$, $\tau= 0.1,~N_x=N_y = 256$ and the initial data $u_0=0.5\sin(x)\sin(y)$ . }\label{fig4}
\end{minipage}
\end{figure}

\subsection{Convergence test}
In this subsection we consider a benchmark computation test with initial data being $u_0=0.5*\sin(x)\sin(y)$. We take {$\nu=1$}, $S=0.1$ and $N_x=N_y= 256$. Then we consider the exact solution $u_e=0.5*e^{-t}\sin(x)\sin(y)$ corresponding to certain forcing term that can be computed explicitly.
 With these settings we perform our numerical experiments with various time steps $\tau = \frac{0.01}{2^k}$ with $k = 0, 1,..., 6$. The relative $L^2$-errors and $L^\infty$-errors at time $T =0.5$ are presented below in Table~\ref{table:2}.  As usual, the experimental order of convergence is computed by comparing the errors of two consecutive refinements. Indeed the rate of convergence indicates the order of the error is $O(\tau)$.

\cmtr{\begin{table}[h!]
\centering
\begin{tabular}{||c || c c ||c c||} 
 \hline
 $\tau=0.01$ & $L^2$-error & Rate & $L^\infty$-error& Rate \\ [0.5ex]
 \hline
 $\tau$ & 9.099e-05 & - & 1.156e-03 &- \\[0.5ex]
 
 $\tau/2$ & 4.523e-05 & 2.012 &5.745e-04& 2.012
 \\[0.5ex]
$\tau/4$ &  2.255e-05& 2.006 &2.864e-04&2.006
 \\[0.5ex]
 $\tau/8$ & 1.126e-05 &2.003  &1.430e-04&2.003
 \\[0.5ex]
 $\tau/16$ & 5.624e-06 & 2.002 & 7.143e-05&2.002
 \\[0.5ex]
 $\tau/32$ & 2.811e-06 & 2.001 &3.570e-05& 2.001\\
 \hline 
\end{tabular}
\caption{Errors and orders of convergence}
\label{table:2}
\end{table}}

\subsection{More dynamics and energy evolution}
In this section, we present more dynamics of our scheme~\ref{1stScheme} with patterns indicating the energy dissipation. In Fig~\ref{fig4} we present the dynamics of 2D Cahn-Hilliard equation using EI scheme \eqref{1stScheme} where $\nu = 0.01$, S=0.1, $\tau= 0.1,~N_x=N_y = 256$ and the initial data $u_0=0.5\sin(x)\sin(y)$. In Fig~\ref{fig5} we present the dynamics choosing $\nu=0.01$, $S=0.1$, $\tau=0.001,~N_x=N_y=256$ and the initial data $u_0=0.05\sin(x)\sin(y)$. In Fig~\ref{fig6} we present the dynamics where $\nu = 0.01$, $S=0.1$, $\tau= 0.01,~N_x=N_y = 256$. The values of the initial data $u_0$ are random between $-1$ and $1$  .

\begin{figure}[!h]
\begin{minipage}[t]{0.48\textwidth}
\centering
\includegraphics[width=0.45\textwidth]{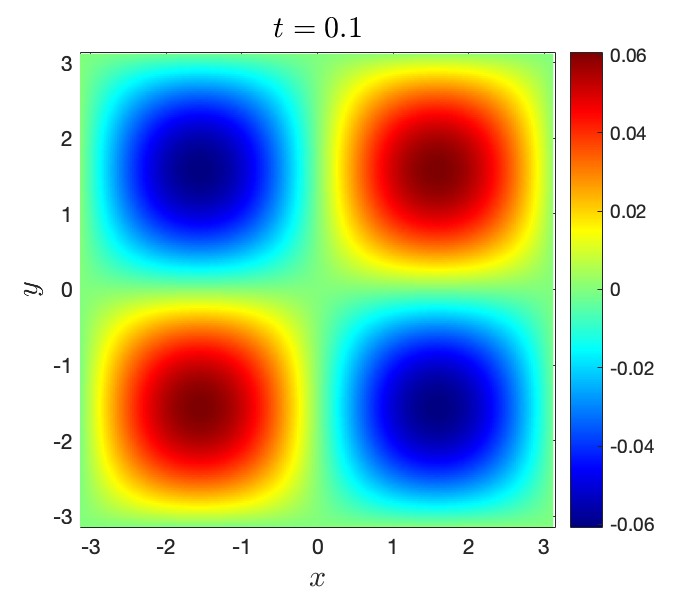}
\includegraphics[width=0.45\textwidth]{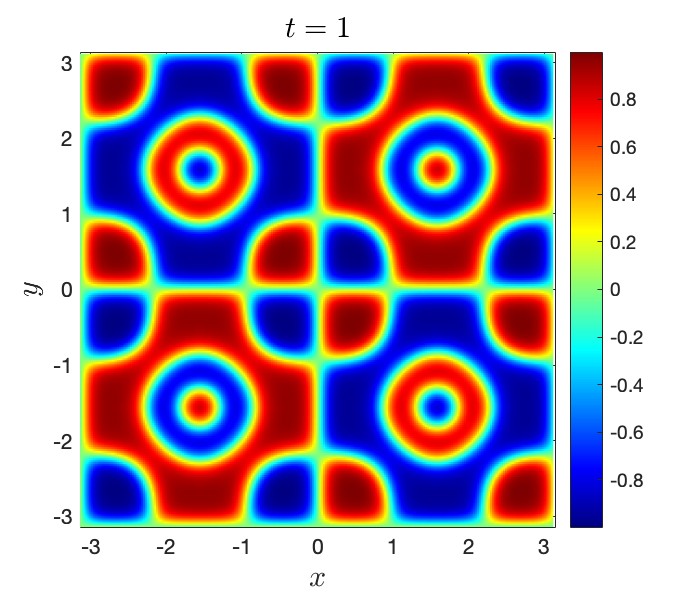}\\
\includegraphics[width=0.45\textwidth]{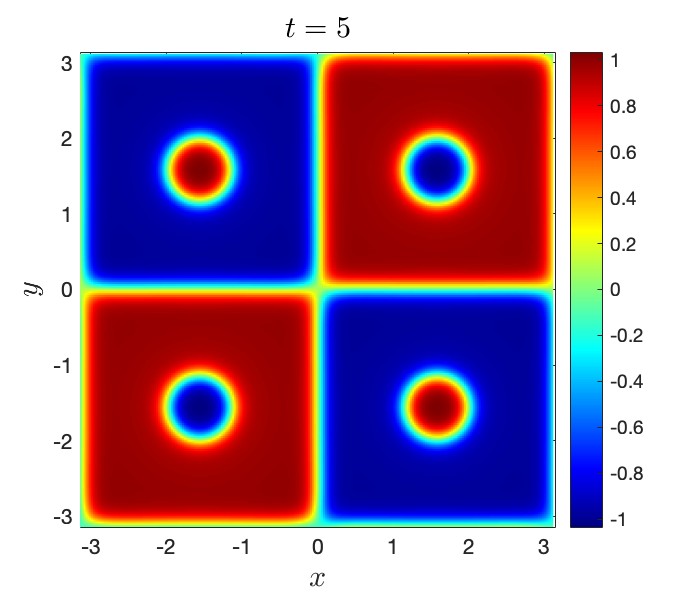}
\includegraphics[width=0.45\textwidth]{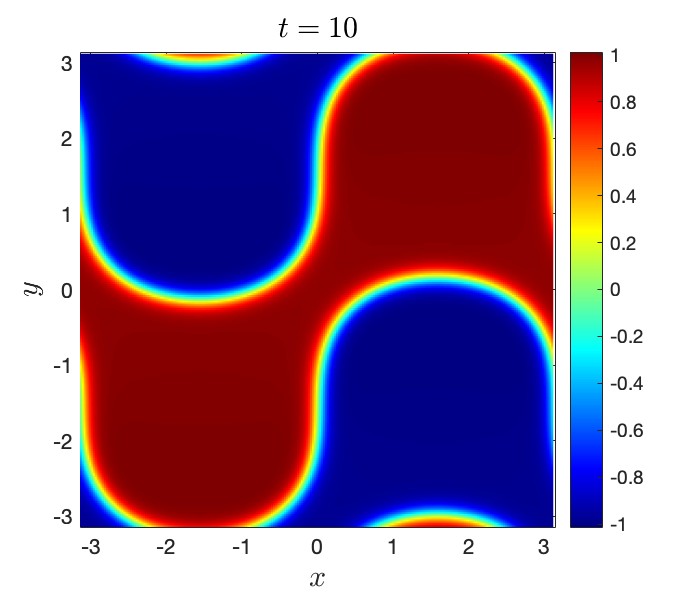}
\includegraphics[width=0.60\textwidth]{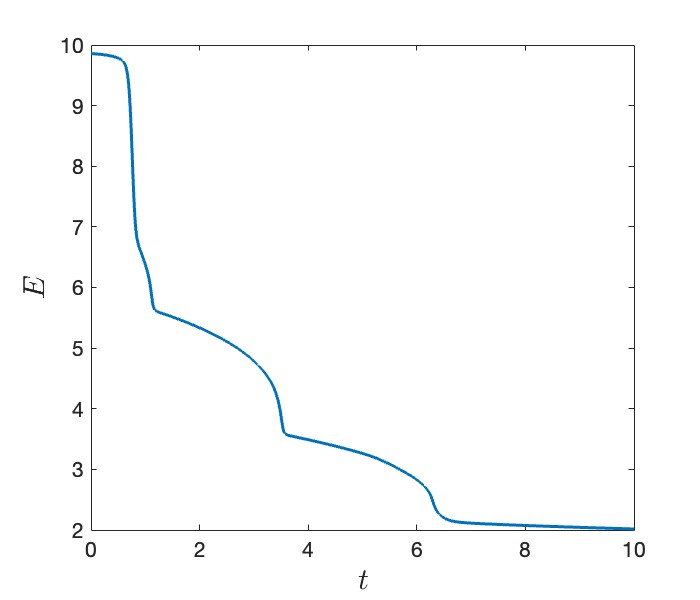}
\caption{\small Dynamics of 2D Cahn-Hilliard equation using EI scheme \eqref{1stScheme} where $\nu = 0.001$, $S=0.1$, $\tau= 0.1,~N_x=N_y = 256$ and the initial data $u_0=0.05\sin(x)\sin(y)$. }\label{fig5}
\end{minipage}
\hfill
\begin{minipage}[t]{0.48\textwidth}
\centering
\includegraphics[width=0.45\textwidth]{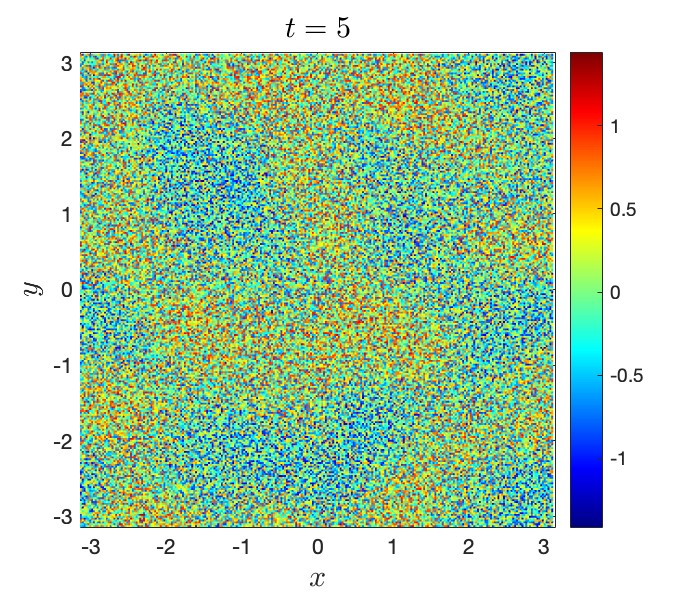}
\includegraphics[width=0.45\textwidth]{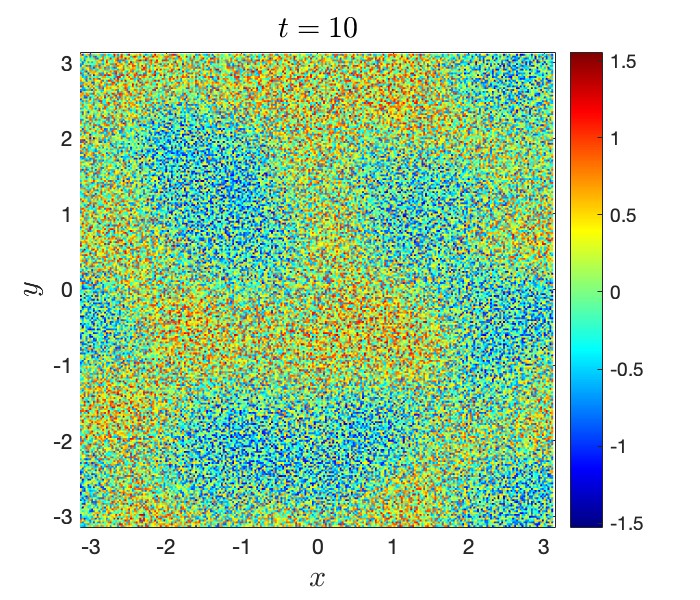}\\
\includegraphics[width=0.45\textwidth]{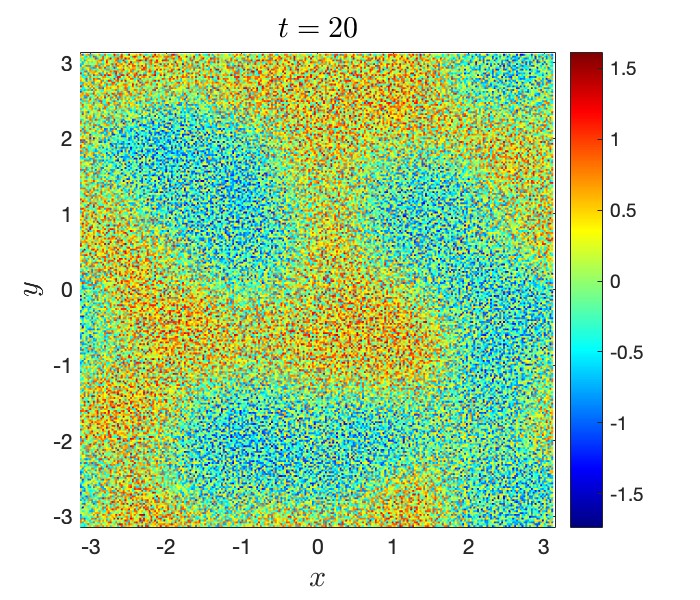}
\includegraphics[width=0.45\textwidth]{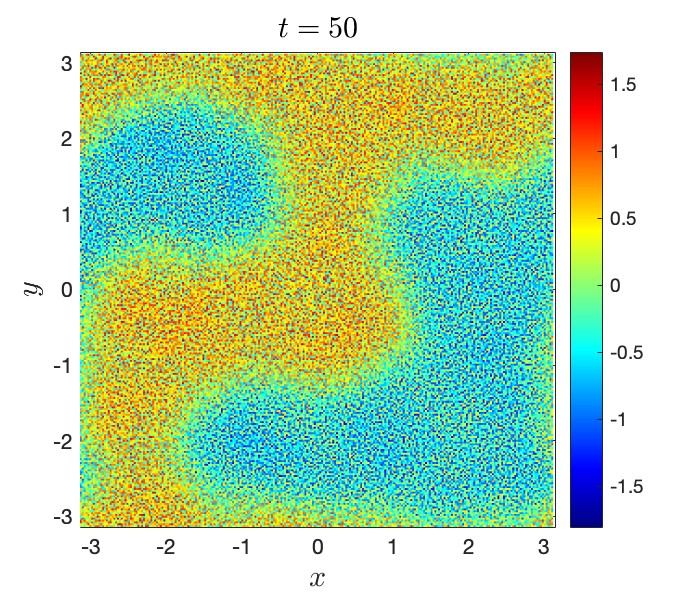}
\includegraphics[width=0.60\textwidth]{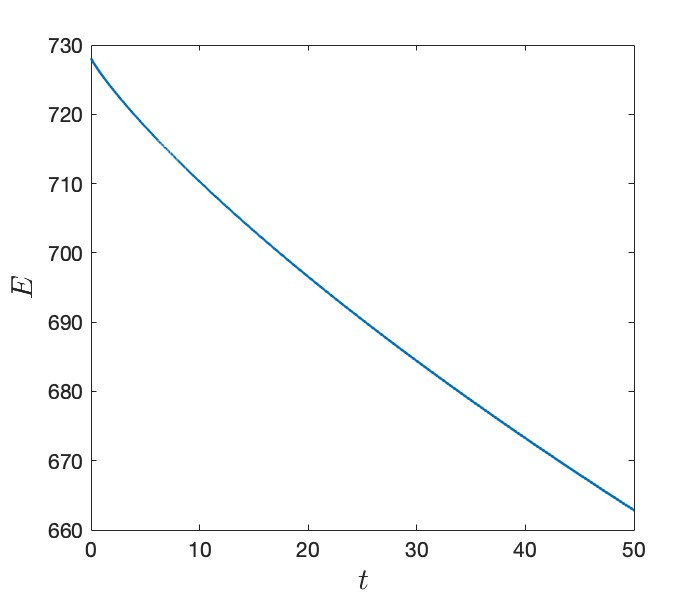}
\caption{\small Dynamics of 2D Cahn-Hilliard equation using EI scheme \eqref{1stScheme} where $\nu = 0.01$, $S=0.1$, $\tau= 0.01,~N_x=N_y = 256$. The values of the initial data $u_0$ are random between $-1$ and $1$. }\label{fig6}
\end{minipage}
\end{figure}

   

\section{Concluding remark}
To conclude, we give a systematic approach on applying EI-type schemes to models where maximum principle no longer exists by solving the Cahn-Hilliard equation with a first order EI scheme and showing the energy dissipation. In fact second order EI schemes will be handled similarly in a subsequent paper. In particular, we prove the energy dissipation of EI schemes for the Cahn-Hilliard equation without assuming any strong Lipschitz condition or $L^\infty$ boundedness. Furthermore, we also analyze the $L^2$ error and present some numerical simulations to demonstrate the dynamics. Indeed, the analysis framework can be applied to more general models and higher order schemes. We leave the discussion in a subsequent work.



\appendix

\section{Proof of Lemma~\ref{H^kregularity_CH}}

\begin{lemma}[$H^k$ boundedness of the exact CH solution]
Assume $u(x,t)$ is a smooth solution to the Cahn-Hilliard equation in $\mathbb{T}^2$ and the initial data $u_0\in H^k(\mathbb{T}^2)$ for $k\geq 2$. Then,
\begin{equation}
\sup_{t\geq 0}\| u(t)\|_{H^k(\mathbb{T}^2)}\lesssim_{k} 1
\end{equation}
where we omit the dependence on $\nu$ and $u_0$. 
\end{lemma}

\begin{proof}The proof is standard and we only sketch the details. To start with we can write the solution $u$ in the mild form
\begin{equation*}
u(t)=e^{-\nu t\Delta^2}u_0+\int_0^te^{-\nu(t-s)\Delta^2}\De(u^3-u)\ ds\ .
\end{equation*}
Note that by the energy dissipation we indeed have $\|u\|_{H^1(\T)}\lesssim 1$ for any $t>0$. Therefore it suffices to this argument inductively, namely we show $\| u\|_{H^2}\lesssim 1 $ for any $t\geq 0$. 
Then by taking any second order spatial derivative and $L^2$ norm in the formula above, we derive
\begin{equation*}
\begin{aligned}
\| D^2u\|_2\leq \| D^2e^{-\nu t\Delta^2}u_0\|_2+\int_0^t\| D^2e^{-\nu(t-s)\Delta^2}\De(u^3-u)\|_2\ ds, 
\end{aligned}
\end{equation*}
where $D^m u$ denotes any differential operator $D^\alpha u$ for any $|\alpha|=m$.

Firstly, we consider the nonlinear part. We indeed observe that
\begin{equation*}
\begin{aligned}
\| D^2e^{-\nu(t-s)\Delta^2}\De(u^3-u)\|_2\lesssim
\|K_1*(u^3-u)\|_{2},
\end{aligned}
\end{equation*}
where $K_1$ is the kernel corresponding to $\De D^2e^{-\nu(t-s)\Delta^2}$. (It is easy to see that here $D$ can interchange with $\De$.) Let $\ga=t-s$, it then suffices to estimate the following quantity:
\begin{align}
    \int_0^t\|K_1*(u^3-u)\|_2\ d\ga= \int_0^1\|K_1*(u^3-u)\|_2\ d\ga+\int_1^t \|K_1*(u^3-u)\|_2\ d\ga.
\end{align}
Here we assume $t>1$ with no loss since the case $t<1$ follows easily as a special case. For the region when $\ga\ge 1$ we get
\begin{equation*}
\begin{aligned}
\|K_1*(u^3-u)\|_2&\lesssim\|K_1*(u^3-u)\|_{\infty}\\
&\lesssim\| K_1\|_2\cdot\| u^3-u\|_2\\
&\lesssim\| K_1\|_2\cdot\| u\|_6^3\\
&\lesssim\|K_1\|_2,
\end{aligned}
\end{equation*}
by the standard Sobolev embedding $\|u\|_6\lesssim \|u\|_{H^1}\lesssim1$.
Note that for $\ga>1$ we have
\begin{equation}
\begin{aligned}
\| K_1\|_2&\lesssim\left(\sum_{|k|\geq 1}|k|^8 e^{-2\nu\ga|k|^4}\right)^{\frac12}\\
&\lesssim\left(\int_1^\infty e^{-2\nu\ga r^4}r^{9}\ dr\right)^{\frac12}\\
&\lesssim \ga^{-\frac12}e^{-\nu\ga}.
\end{aligned}
\end{equation}
As a result, we have
\begin{align}
    \int_1^t \|K_1*(u^3-u)\|_2\ d\ga\lesssim \int_1^t \ga^{-\frac12}e^{-\nu\ga}\ d\ga\lesssim 1.
\end{align}
The arguments above indeed indicate a smoothing effect. The region $0<\ga<1$ follows from the standard local theory.
This shows $\int_0^t\| De^{\nu (t-s)\Delta}u\|_2\ ds\lesssim 1$; the linear part follows from similar arguments. As a remark our proof works at least for dimension $d=1,2,3$.

\end{proof}

\section*{Acknowledgement}
X. Cheng was partially supported by the Shanghai ``Super Postdoc" Incentive Plan (No. 2021014), the International
Postdoctoral Exchange Fellowship Program (No. YJ20220071) and China Postdoctoral Science Foundation (Grant No. 2022M710796, 2022T150139). 

\bibliographystyle{abbrv}

\end{document}